\documentclass[11pt]{amsart}
\usepackage[all]{xy}
\usepackage{color}
\usepackage{hyperref}
\definecolor{tocolor}{rgb}{.1,.1,.1}
\definecolor{urlcolor}{rgb}{.2,.2,.6}
\definecolor{linkcolor}{rgb}{.1,.1,.3}
\definecolor{citecolor}{rgb}{.3,.2,.1}
\hypersetup{backref=true, colorlinks=true, urlcolor=urlcolor, linkcolor=linkcolor, citecolor=citecolor}

\usepackage{amssymb}
\newtheorem{theorem}{Theorem}[section]

\newtheorem{proposition}[theorem]{Proposition}

\newtheorem{lemma}[theorem]{Lemma}
\theoremstyle{definition}    
\newtheorem{definition}[theorem]{Definition}
\theoremstyle{remark}

\newtheorem{remark}[theorem]{Remark}
\newtheorem{remarks}[theorem]{Remarks}
\newtheorem{example}[theorem]{Example}

\newcommand\A{\mathcal{A}}
\renewcommand{\AA}{\mathbb{A}}
\newcommand{\Cour}[1]      {[\![#1]\!]}
\newcommand\M{\mathcal{M}}

\renewcommand{\L}{\mathcal{L}}
\newcommand{\LL}{\mathsf{L}}
\newcommand{\HH}{\mathsf{H}}
\renewcommand{\O}{\mathcal{O}}
\newcommand{\T}{\mathbb{T}}

\newcommand{\ca}{\mathcal}

\newcommand{\U}{\on{U}}
\newcommand{\E}{\ca{E}}

\newcommand{\F}{\mathcal{F}}

\newcommand{\R}{\mathbb{R}}

\newcommand{\Z}{\mathbb{Z}}

\newcommand\pt{\on{pt}}
\renewcommand{\P}{\ca{P}}

\newcommand\lie[1]{\mathfrak{#1}}

\newcommand{\g}{\lie{g}}
\newcommand{\dd}{\lie{d}}

\renewcommand{\a}{\mathsf{a}}

\renewcommand{\u}{\lie{u}}
\newcommand{\on}{\operatorname}

\newcommand{\Aut}{ \on{Aut} } 
\newcommand{\Map}{\on{Map}}
\newcommand{\Ad}{ \on{Ad} }
\newcommand{\ad}{\on{ad}}
\newcommand{\Hol}{ \on{Hol} }

\renewcommand{\ker}{ \on{ker}}
\newcommand{\ran}{\on{ran}}

\newcommand{\gr}{\on{gr}}
\newcommand{\Mult}{  \on{Mult}}

\newcommand{\red}{{\on{red}}}

\newcommand{\da}{\dasharrow}

\newcommand\qu{/\kern-.7ex/} 
\newcommand{\TG}{\mathbb{T}G}

\newcommand{\hra}{\hookrightarrow}

\renewcommand{\d}{{\mbox{d}}}
\newcommand{\ol}{\overline}

\renewcommand{\S}{\mathcal{S}}
\newcommand\sig{\sigma}

\newcommand\om{\omega}

\newcommand{\f}{\frac}

\newcommand{\p}{\partial}
\renewcommand{\l}{\langle}
\renewcommand{\r}{\rangle}
\newcommand{\rra}{\rightrightarrows}
\newcommand\hh{{\f{1}{2}}}

\newcommand{\ti}{\tilde}
\newcommand{\eeq}{\end{eqnarray*}}

\renewcommand{\H}{\ca{H}}

\newcommand{\pr}{\on{pr}}
\newcommand{\wh}{\widehat}
\newcommand{\wt}{\widetilde}
\newcommand{\mf}{\mathfrak}
\renewcommand{\subset}{\subseteq}

\renewcommand{\star}{*}

\newcommand{\II}{\mathrm{I}}

\def\bd{{\partial}}

\def\s{{\mathfrak{s}}}
\def\Inv{\on{Inv}}
\def\DD{\mathcal{T}}

\long\def\authornote#1{%
        \leavevmode\unskip\raisebox{-3.5pt}{\rlap{$\scriptstyle\diamond$}}
        \marginpar{\centering 
       \hbadness=10000 
        \def\baselinestretch{0.7}\tiny 
        \it #1\par}}

\newcommand{\eckhard}[1]{\authornote{EM: #1}}
\newcommand{\alejandro}[1]{}	

\begin{document}
\sloppy
\title{Dirac geometry of the holonomy fibration}
\date{\today}
\author{A. Cabrera}
\author{M. Gualtieri}
\author{E. Meinrenken}
\maketitle
\tableofcontents

\section{Introduction}

An invariant Poisson structure on a finite-dimensional principal
bundle $P\to B$ descends to a Poisson structure on the base. This is
immediate from the identification of functions on $B$ with invariant
functions on $P$, or alternatively, because the invariant Poisson bivector field
$\pi_P$ pushes down to a Poisson bivector field on $B$.

One is tempted to apply these facts to the following
infinite-dimensional setting.  Let $G$ be a connected Lie group. Its
loop group $LG=\on{Map}(S^1,G)$ acts by gauge transformations on the space
\[ \A=\Omega^1(S^1,\g) \]
of connections on the trivial $G$-bundle over the circle. The based
loop group $L_0G\subseteq LG$ acts freely, and the holonomy of a
connection identifies $\A/L_0G$ with $G$. We will refer to the resulting
principal $L_0G$-bundle
\[ \on{Hol}\colon \A\to G\]
as the \emph{holonomy fibration}. Suppose the Lie algebra carries an
invariant metric, used to identify $\g$ with
$\g^\star$. It defines a central extension $\wh{L\g}$ of $L\g$ by
$\R$, and one may regard $\A$ as the affine subspace of
$\wh{L\g}^\star$ at level $1$. Formally, it carries a Poisson
structure called the Lie-Poisson structure, with symplectic
leaves the level 1 coadjoint orbits of $LG$.

The naive attempt to push this down to a Poisson structure on $G$ runs
into problems, related to the precise meaning of a Poisson structure
in infinite dimensions. Indeed, the Lie-Poisson structure on
$\A$, viewed as a bilinear bracket $\{\cdot,\cdot\}$ on functions,
cannot be defined on \emph{all} functions; its domain does not even
contain all pullbacks $\Hol^* f$ with $f\in
C^\infty(G)$. Similarly, the Lie-Poisson structure on $\A$
cannot be a genuine bivector field, since sections of $\wedge^2 T\A$,
by definition, have only finite rank.

In this paper, we shall take a third viewpoint, regarding the
Lie-Poisson structure on $\A$ as a \emph{Dirac structure}. Recall
that a Dirac structure on a finite-dimensional manifold $Q$ is a
Lagrangian sub-bundle $E\subseteq TQ\oplus T^\star Q$ satisfying a certain
integrability condition. Poisson structures are Dirac structures for
which $E$ is the graph of a skew-adjoint bundle map $T^\star Q\to TQ$. In
finite dimensions, this is equivalent to the property $E\cap TQ=0$.

The definition of Dirac structures carries over to
infinite-dimensional Hilbert manifolds, but here the conditions of $E$
being a graph or having trivial intersection with the tangent bundle
are no longer equivalent. We will call a Dirac structure $E$ with the
latter property a \emph{weak Poisson structure}. Equivalently, the weak Poisson structures 
are described as a family of skew-adjoint operators $D_q\colon \on{dom}(D_q)\to T_qQ$, with dense domain in $T^\star_qQ$. 
The leaves of a weak Poisson structure carry closed 2-forms that are weakly
symplectic. Taking $\A$ to consist of connections of a fixed Sobolev
class (e.g. $L^2$ or higher),  we observe that the Lie-Poisson structure is well-defined 
as a weak Poisson structure in the above sense. The corresponding skew-adjoint operators are the covariant derivatives $\partial_A$. 

Using the reduction procedure for Dirac structures \cite{bur:red}, this weak Poisson structure may be pushed down under the map $\Hol$.  We will show that the result is the well-known \emph{Cartan-Dirac structure} on $G$. The Cartan-Dirac structure had been discovered independently by Alekseev, \v{S}evera and Strobl in the late 1990s, and plays an important role in the theory of $D$-branes \cite{car:fu,st:db,kli:wzw} as well as for quasi-Hamiltonian $G$-spaces \cite{al:pur,al:mom}. Our reduction procedure extends to Hamiltonian spaces, and clarifies the
correspondence \cite{al:mom} between Hamiltonian loop group spaces \cite{me:lo} and
q-Hamiltonian $G$-spaces. We also describe multiplicative properties of the Cartan-Dirac structure \cite{al:pur,lib:dir} from the point of view of reduction from suitable spaces of connections.

In this article, we will mostly work with a closely related holonomy fibration $\Hol\colon \A_\II\to G$, given by connections on the interval $\II=[0,1]$, with an action of the gauge group 
$G_\II=\on{Map}(\II,G)$. The `Lie-Poisson' structure on $\A_\II$ is a Dirac structure described by connections $\partial_A$ as before, but whose domain involves periodic boundary conditions. The reduction by the group $G_{\II,\partial\II}$ of  gauge transformations trivial at the boundary, results in the Cartan-Dirac structure. 
From this point of view, we may consider alternative boundary conditions for the family of operators $\partial_A$, given by Lagrangian Lie subalgebras $\s\subset \ol\g\oplus\g$. The corresponding weak Poisson structures on $\A_\II$ reduce to generalizations of the Cartan-Dirac structure.


\vskip.1in
\noindent{\bf Acknowledgements.} 
A.~C. thanks University of Toronto for hospitality during the beginning of this project.
M.~G. was supported by an NSERC Discovery Grant and acknowledges support from U.S. National Science Foundation grants DMS 1107452, 1107263, 1107367 ``RNMS: GEometric structures And Representation varieties'' (the GEAR Network).  
E.~M. was supported by an NSERC Discovery Grant. 
We thank Henrique Bursztyn for helpful discussions and for posing the problem of determining the geometric nature of the reduction of Hamiltonian loop group spaces to quasi-Hamiltonian spaces.

%
\section{Dirac structures in infinite dimensions}
In this section, we review the theory of Courant algebroids, Dirac structures, and their reduction in an infinite-dimensional context. For a
treatment of differential geometry on Banach manifolds and Hilbert
manifolds, see e.g.~\cite{ab:ma}.

Much of the material is a direct extension of the finite-dimensional
theory. Special care needs to be taken due to the fact that the sum of
closed subspaces of a Banach space need not be closed. These problems
are already apparent in the linear version of the theory, described below.

\subsection{Linear Dirac geometry in infinite dimensions}
\label{subsec:technical}
Throughout this paper, the terms \emph{Banach space} and \emph{Hilbert
  space} designate a real topological vector space $V$ whose topology is
defined by a Banach norm and Hilbert inner product, respectively.  The
norm or inner product itself is not considered part of the
structure. 
By \cite{lin:com}, a Banach space is a Hilbert space if and only if
every closed subspace admits a closed complement. For this reason, we
will mainly work with Hilbert spaces and Hilbert manifolds.

A continuous symmetric bilinear form $\l\cdot,\cdot\r\colon V\times
V\to \R$ on a Hilbert space $V$ is called \emph{non-degenerate} if the associated map $V\to V^\star ,\
v\mapsto \l v,\cdot\r$ is an isomorphism. We will
refer to $\l\cdot,\cdot\r$ as a pseudo-Riemannian metric, or simply as
a \emph{metric}, and call $V$ a \emph{metrized Hilbert space}. We
stress that $\l\cdot,\cdot\r$ is not necessarily a Hilbert space inner
product.

If $F$ is a subspace of a metrized Hilbert space $V$, denote by
$F^\perp$ its orthogonal relative to the metric.  Accordingly,
$F$ is called \emph{isotropic} if $F\subseteq F^\perp$,
\emph{co-isotropic} if $F^\perp\subseteq F$, and \emph{Lagrangian} if
$F=F^\perp$. A \emph{Lagrangian splitting} of $V$ is a direct sum
decomposition $V=F_1\oplus F_2$ into Lagrangian subspaces. In finite
dimensions, this is equivalent to $F_1\cap F_2=0$, but in infinite
dimensions this is stronger:
\begin{example}\label{ex:hh1}
  Suppose $\H$ is a Hilbert space, and equip $V=\H\oplus \H^\star$
  with the metric 
  \[ 
  \l v_1+\mu_1,\,v_2+\mu_2\r=\mu_2(v_1)+\mu_1(v_2)
  \] 
  for $v_1,v_2\in \H,\ \mu_1,\mu_2\in \H^\star $. Then $\H,\,\H^\star$
  are Lagrangian subspaces.
  Suppose 
  \[ 
  A\colon \on{dom}(A)\to \H^\star 
  \] 
  is an unbounded linear operator with dense domain
  $\on{dom}(A)\subseteq \H$. By definition, $A$ is a closed operator if
  and only if its graph $\gr(A)$ is closed, and is an unbounded
  skew-adjoint operator if and only if $\on{gr}(A)$ is Lagrangian. 
  Now suppose that $A$ is an unbounded skew-adjoint operator with 
  $\on{dom}(A)\neq \H$. Then $\gr(A),\ \H^\star $
  are Lagrangian subspaces with trivial intersection, but
  $\gr(A)+\H^\star \neq \H\oplus \H^\star$.
\end{example}
If $C\subseteq V$ is a closed co-isotropic subspace of a metrized
Hilbert space, we define a \emph{reduced space} $V_C=C/C^\perp$. It
inherits a metric from the metric on $V$. Given a subspace
$F\subseteq V$, define $F_C=(F\cap C)/(F\cap C^\perp)$. In finite
dimensions, the reduction $L_C$ of a Lagrangian subspace $L$ is again
Lagrangian, but this need not be the case in infinite dimensions:
\begin{example}\label{ex:hh}
  In the setting of Example \ref{ex:hh1}, pick $v\in \H-\on{dom}(A)$,
  and let $C=\on{span}(v)\oplus \H^\star $. Then $C$ is coisotropic,
  with $C^\perp=\on{ann}(v)\subseteq \H^\star $. Hence
  $C/C^\perp=\on{span}(v)\oplus \on{span}(v)^\star $ is
  2-dimensional. The Lagrangian subspace $L=\gr(A)$ satisfies
  $L\cap C=0$, hence $L_C=0$ is not Lagrangian.
\end{example}
To ensure that the reduction of a Lagrangian subspace is Lagrangian,
we need an additional condition:
\begin{proposition}\label{prop:linred1}
  Let $V$ be a metrized Hilbert space, and
  $C$ a closed co-isotropic subspace of $V$. Let $L\subseteq V$ be a
  Lagrangian subspace with the property that $L+C$ is closed. Then
  $L_C=(L\cap C)/(L\cap C^\perp)$ is Lagrangian in $V_C$.
\end{proposition}
A proof is given in the Appendix, see Proposition
\ref{prop:linred}.
\begin{remark}\label{rem:ortrem}
  Given a metrized Hilbert space $V$, the sum $F_1+F_2$ of subspaces
  is closed if and only if $F_1^\perp+F_2^\perp$ is closed.  Hence,
  the condition in Proposition \ref{prop:linred1} is equivalent to the
  condition that $L+C^\perp$ be closed.
\end{remark}

\begin{remark}
  In subsequent sections, we use vector bundle versions of the
  results described above. We refer to a Hilbert vector bundle $V\to
  M$ over a Hilbert manifold, with a (pseudo-Riemannian) fiber metric
  $\l\cdot,\cdot\r$, as a \emph{metrized vector bundle}. Given a
  closed coisotropic subbundle $C\subseteq V$, the quotient
  $V_C=C/C^\perp$ inherits a metric.  For a Lagrangian sub-bundle
  $L\subseteq V$ the reduction $L_C=(L\cap C)/(L\cap C^\perp)$ is a
  Lagrangian subbundle provided $L+C$ is a closed subbundle. In
  particular, this is the case if the intersection is
  \emph{transverse}, i.e. $L+C=V$, or if $L\subseteq C$.
\end{remark}

For any metrized Hilbert space $V$, let
$\ol{V}$ denote the same Hilbert space with the opposite metric. A
\emph{Lagrangian relation}
\[ R\colon V_1\da V_2\]
between two metrized Hilbert spaces is a linear relation whose graph 
$\gr(R)\subseteq
V_2\times \ol{V}_1$ is Lagrangian. We will write $v_1\sim_R v_2$ if and only if 
$(v_2,v_1)\in \gr(R)$, and define the kernel and range of $R$ as 
\[ \ker(R)=\{v_1\in V_1|\ v_1\sim_R 0\},\  \ \ \ 
\ran(R)=\{v_2\in V_2|\ \exists v_1\in V_1\colon v_1\sim_R v_2\}.\]
The space $\ker(R)$ is closed, but $\ran(R)$ not necessarily so. 
Similarly, we define $\ker^\star(R)=\ker(R^\top)$ and $\on{ran}^\star(R)=\ran(R^\top)$, where 
$R^\top\colon V_2\to V_1$ is the transpose relation. We have 
\[ \ker(R)=\ran^\star(R)^\perp,\ \ \ker^\star(R)=\ran(R)^\perp.\]

Given another Lagrangian relation $R'\colon
V_2\da V_3$, one defines $R'\circ R$ as a composition of relations.  If the
$V_i$ are finite-dimensional, then $R'\circ R$ is again a Lagrangian
relation, but in infinite dimensions additional assumptions are
needed. We say that $R',R$ have \emph{transverse composition} if 
\begin{equation}\label{eq:strongtrans}
 \ran(R)+\ran^\star (R')=V_2.
\end{equation}
\begin{proposition}
If $R',\ R$ have transverse composition, then $R'\circ R$ is a Lagrangian relation. 
\end{proposition}
\begin{proof}
Let $V=(V_3\times\ol{V}_2)\times (V_2\times\ol{V}_1)$ and $C=V_3\times (V_2)_\Delta\times \ol{V_1}$, where 
$(V_2)_\Delta\subseteq V_2\times \ol{V}_2$ is the diagonal subspace. Then 
\[ \gr(R'\circ R)=(\gr(R')\times \gr(R))_C\subseteq
V_C=V_3\times
\ol{V}_1.\]
Since $\ran(R)+\ran^\star (R')$ is the  image of $\gr(R')\times \gr(R)$ under the projection $V\to V/C\cong V_2$, this is equivalent to  
$(\gr(R')\times \gr(R))+C=V$. By Proposition \ref{prop:linred1} this guarantees that $R'\circ R$ is a Lagrangian relation.
\end{proof} 
Taking orthogonals, we see that the transversality \eqref{eq:strongtrans}
implies 
\begin{equation}\label{eq:weaktrans}
 \ker(R')\cap \ker^\star(R)=0,
\end{equation}
which says that whenever $v_1\sim_{R'\circ R} v_3$, then the element $v_2$ with 
$v_1\sim_R v_2$ and $v_2\sim_{R'} v_3$ 
is uniquely determined. We will call the composition $R'\circ R$ \emph{weakly transverse} if the condition \eqref{eq:weaktrans} holds, or equivalently $\ran(R)+\ran^\star (R')$ is dense in $V_2$. 
\begin{definition}
A pair $(V,E)$ consisting of a metrized Hilbert space and a Lagrangian subspace is called a \emph{linear Dirac structure}. A \emph{linear Dirac morphism}  $R\colon (V_1,E_1)\da (V_2,E_2)$ is a Lagrangian relation $R\colon V_1\da V_2$ such that $E_2=R\circ E_1$,
where the composition is weakly transverse (i.e.~ $E_1\cap \ker(R)=0$). If the composition is transverse
(i.e.~$E_1+\ran^\star (R)=V_1$), we will call $R$ a \emph{strong linear Dirac morphism}.
\end{definition}
Here the Lagrangian subspaces $E_i\subseteq V_i$ are regarded as linear relations 
$E_i\colon 0\da V_i$. In the following result, we consider $F_i\subseteq V_i$ as Lagrangian relations 
$F_i\colon V_i\da 0$. 

\begin{proposition}\label{prop:forback}
  Suppose $R\colon (V_1,E_1)\da (V_2,E_2)$ is a linear Dirac morphism, and let $F_2$ be a Lagrangian 
complement  to $E_2$. Then $F_1=F_2\circ R$ is a Lagrangian subspace with $E_1\cap F_1=0$. 
  If $R$ is a strong linear Dirac morphism, then $V_1=E_1\oplus F_1$.
\end{proposition}
\begin{proof}
Since $E_2\subseteq \ran(R)$, we have that $\ran(R)+F_2=V_2$. Hence the composition is transverse, and $F_1=F_2\circ R$ is a Lagrangian subspace. Suppose $x_1\in E_1\cap F_1$. Since 
$x_1\in F_1$, there exists $x_2\in F_2$ with $x_1\sim_R x_2$. Since $x_1\in E_1$, this relation implies that $x_2\in E_2$. Hence $x_2=0$. But $x_1\in E_1,\ x_1\sim_R 0$ means $x_1=0$, 
by weak transversality of the composition $R\circ E_1$. 

Suppose now that the composition is transverse, so that $V_1=E_1+\ran^\star
  (R)$. Let $x_1\in \ran^\star (R)$, so that $x_1\sim_R x_2$ for some $x_2\in V_2$. 
     Write $x_2=x_2'+x_2''$
  with $x_2'\in E_2$ and $x_2''\in F_2$. Let $x_1'\in E_1$ be an
  element with $x_1'\sim_R x_2'$, and put $x_1''=x_1-x_1'$. Then
  $x_1''\sim_{R} x_2''$, hence $x_1''\in F_1$. This shows $V_1=E_1\oplus F_1$. 
\end{proof}

\begin{proposition}\label{prop:automatic}
  Suppose $R\colon (V_1,E_1)\da (V_2,E_2)$ and $R'\colon (V_2,E_2)\da
  (V_3,E_3)$ are strong linear Dirac morphisms.  Then 
  the composition $R'\circ R$ is transverse, and defines a 
  strong linear Dirac morphism
  $R'\circ  R\colon (V_1,E_1)\da (V_3,E_3)$.
\end{proposition}
\begin{proof}
  Choose a Lagrangian complement $F_3$ to $E_3\subseteq V_3$. Then
  $F_2=F_3\circ R'\subseteq V_2$ is a Lagrangian complement to $E_2$,
  and $F_1=F_2\circ R$ is a Lagrangian complement to $E_1$. We have
  $E_2\subseteq \ran(R)$ and $F_2\subseteq \ran^\star (R')$, hence
  $\ran(R)+\ran^\star (R')=V_2$, proving transversality of the
  composition $R'\circ R$. Similarly, $F_1=F_3\circ (R'\circ R)$ shows
  that $F_1\subseteq \ran^\star(R'\circ R)$. Hence $E_1+\ran^\star(R'\circ
  R)=V_1$.
\end{proof}

%

\subsection{Courant algebroids}
The usual definition of a Courant algebroid \cite{liu:ma,roy:co} works
equally well for infinite dimensional manifolds. In the remainder of this
section we shall use the terms ``manifold'', ``vector bundle'', ``Lie
group'', etc. to refer to Hilbert manifold, Hilbert vector bundle, Hilbert
Lie group, and so on. A \emph{metrized vector bundle} is a Hilbert vector bundle with a fiberwise (pseudo-Riemannian) metric. 

A \emph{Courant algebroid} is a metrized vector bundle
$(\AA,\l\cdot,\cdot\r)$ over a manifold $Q$, equipped with a smooth bundle map
$\a\colon \AA\to TQ$ called the \emph{anchor}, and a bilinear
\emph{Courant bracket} $\Cour{\cdot,\cdot}\colon
\Gamma(\AA)\times\Gamma(\AA)\to \Gamma(\AA)$, such that the following
axioms are satisfied, for all smooth sections $\sig_1,\sig_2,\sig_3$
of $\AA$:
\begin{equation}
\begin{aligned}\label{courdef}
  \Cour{\sig_1,\Cour{\sig_2,\sig_3}}&=
     \Cour{\Cour{\sig_1,\sig_2},\sig_3}+\Cour{\sig_2,\Cour{\sig_1,\sig_3}},\\
  \a(\sig_1) \l\sig_2,\sig_3\r&=
     \l \Cour{\sig_1,\sig_2},\sig_3\r+\l \sig_2,\Cour{\sig_1,\sig_3}\r,\\
  \a^\star\d\l\sig_1,\sig_2\r&=
     \Cour{\sig_1,\sig_2}+\Cour{\sig_2,\sig_1}.
\end{aligned}
\end{equation}
Here $\a^\star\colon T^\star Q\to \AA$ is the dual anchor composed with the isomorphism $\AA^{*}\cong \AA$ given by the metric. 

These axioms imply the following properties \cite{uch:rem}, for all $f\in C^\infty(Q)$: 
%
%
\begin{equation*}
\begin{aligned}
\Cour{\sig_1,f\sig_2}&=f\Cour{\sig_1,\sig_2}+(\a(\sig_1)f)\sigma_2,\\
\a(\Cour{\sig_1,\sig_2})&=[\a(\sig_1),\a(\sig_2)].
\end{aligned}
\end{equation*}
%
%
A \emph{Dirac structure} $(\AA,E)$ on $Q$ is a Courant algebroid together with a Lagrangian subbundle $E\subseteq \AA$ whose space of sections is closed under the bracket. If the Courant algebroid 
$\AA$ is fixed, we refer to $E$ itself as the Dirac structure. For any Dirac structure, the Courant bracket restricts to a Lie bracket on $\Gamma(E)$, thus $E$ is a Lie algebroid.  A connected submanifold $\O\subseteq Q$ is called a \emph{leaf} of $E$ if $\a(E|_\O)=T\O$, and is maximal with this property.  If $\dim Q<\infty$,  the Stefan-Sussmann theorem \cite{bal:not,ste:int,sus:orb}
asserts that $Q$ acquires a singular foliation by leaves. In  infinite dimensions, there are similar 
results due to  Chillingworth-Stefan \cite{chi:int} and Pelletier \cite{pel:int} (the latter reference discusses foliations defined by Banach-Lie algebroids).  In our main applications the foliation will be explicitly given as the orbits of a Lie group action. 

\begin{example}\mbox{}

\begin{enumerate}
\item 
Suppose $Q$ is a manifold with a closed 3-form $\eta\in\Omega^3(Q)$. Then
the direct sum $TQ\oplus T^\star Q$ carries the structure of a Courant
algebroid, with metric $\l v_1+\mu_1,v_2+\mu_2\r=\l\mu_1,v_2\r+\l\mu_2,v_1\r$, 
with anchor the projection to the first summand, and with the Courant 
bracket
\[
  \Cour{v_1+\mu_1,v_2+\mu_2}=[v_1,v_2]+\L_{v_1}\mu_2-\iota_{v_2}\d\mu_1
  +\iota_{v_1}\iota_{v_2}\eta.
\]
We will denote this Courant algebroid by $\T Q_\eta$. If $\eta=0$, it
is called the \emph{standard Courant algebroid} and is denoted $\T
Q$. 
Suppose $E\subseteq \T Q_\eta$ is a Dirac structure.  If $i_\O\colon\,\O\to Q$ is 
the inclusion of a leaf of $E$, then there is a 2-form $\om_\O\in
\Omega^2(\O)$, uniquely defined by the property
$\om_\O(v_1,v_2)=\l\alpha_1,v_2\r$ for all $v_i\in T_m\O$, where
$\alpha_1\in T^\star Q$ is chosen so that $v_1+\alpha_1\in E_m$. It follows that 
the 2-form satisfies $\d\om_\O=-i_\O^*\eta$.
\item \label{ex:coiso}
Suppose $\dd$ is a Lie algebra with an
  invariant metric. Given a $\dd$-action on $Q$ such that the stabilizer
  algebras $\dd_m=\{\xi\in\dd|\ \xi_Q(m)=0\}$ are coisotropic, the
  product
  \[ \AA=Q\times\dd\]
  becomes a Courant algebroid, with anchor the 
  action map $Q\times \dd\to TQ$, and with Courant bracket
  extending the Lie bracket on constant sections (see
  \cite{lib:cou}). This is called an \emph{action Courant algebroid}. For any Lagrangian Lie subalgebra $\s\subset \dd$, the subbundle $E=Q\times \s$ is a Dirac structure in $\AA$. 
\end{enumerate}
\end{example}

\subsection{Weak Poisson structures}\label{subsec:weakpoisson}
  A Lagrangian subbundle $E\subseteq \T Q$ with the property $E\oplus
  TQ=\T Q$ amounts to a continuous skew-symmetric bilinear form $\pi$
  on $T^\star Q$, such that $E=\gr(\pi^\sharp)$ is the graph of
  the associated map. If $E$ is a Dirac structure with this property,
  we will call $\pi$ (or $E$ itself) a \emph{Poisson structure} on
  $Q$. In particular, $\pi$ determines a bracket on $C^\infty(Q)$ in
  the usual way. For general Banach (as opposed to Hilbert) manifolds, the definition is more
  involved, see Odzijewicz-Ratiu \cite{odz:ban}.  Given a leaf $\O$  of a Poisson structure, the 2-form $\om$  on that leaf is symplectic, in the strong sense that the bundle map
  $\Omega^\flat\colon TQ\to T^\star Q$ is invertible. 
  
  A Dirac structure  $E\subseteq \T Q$ satisfying the weaker condition $E\cap TQ=0$ will be
  called a \emph{weak Poisson structure}; this may be regarded as a family of skew-adjoint \emph{unbounded} operators. The resulting  2-forms $\om$ on
  leaves are only weakly symplectic, in the sense that $\Omega^\flat$ is
  injective.  In the finite-dimensional setting, the notions
  coincide.  See Posthuma \cite[Chapter 4.1]{po:qu} for another definition of weak Poisson structure. 
 Given a weak Poisson structure $E$, let $C^\infty_E(Q)$  be the space of smooth functions $f$ 
 for which there exists a vector field $v_f$ with $v_f+\d f\in \Gamma(E)$. Since $E\cap TQ=0$, 
 the vector field $v_f$ is uniquely determined.  The elements of $C^\infty_E(Q)$ are called \emph{admissible} \cite{cou:di} or \emph{Hamiltonian} \cite{ab:ma} functions, and $v_f$ the corresponding \emph{Hamiltonian vector field}. The space of Hamiltonian functions is a Poisson algebra for the bracket 
\[ \{f_1,f_2\}=v_{f_1}(f_2).\]   
 %

\subsection{Morphisms}\label{subsec:morphisms}
Morphisms of Courant algebroids and Dirac structures are defined as Lagrangian correspondences. 

For any Courant algebroid $\AA$, denote by $\ol{\AA}$ the Courant algebroid which is obtained from from $\AA$ by reversing
the sign of the metric.
A \emph{morphism of Courant algebroids}
\[ R\colon \AA_1\da \AA_2\]
is a smooth map $\Phi\colon Q_1\to Q_2$ of the base manifolds,
together with a Lagrangian subbundle $\gr(R)\subseteq \AA_2\times\ol{\AA}_1$ along
the graph $\gr(\Phi)\subseteq Q_2\times Q_1$, satisfying the following integrability 
condition:  If two sections of $\AA_2\times\ol{\AA}_1$ restrict to sections 
of $\gr(R)$, then so does their Courant bracket. We will depict Courant morphisms as 
follows
\begin{equation}\label{eq:R}  \xymatrix{ \AA_1 \ar@{-->}[r]^{R}\ar[d] & \AA_2 \ar[d] \\ Q_1\ar[r]_\Phi & Q_2
}\end{equation}
Composition of Courant morphisms
is defined as a composition of Lagrangian relations, assuming that the
composition is transverse. As shown in \cite{lib:dir}, the integrability condition is preserved under composition.

For $x_i\in \AA_i$, we will write $x_1\sim_R x_2$ if $(x_2,x_1)\in \gr(R)$.
Similarly, if $\sig_i\in\Gamma(\AA_i)$ are sections we write
$\sig_1\sim_R \sig_2$ if $(\sig_2,\sig_1)$ restricts to a section of
$\gr(R)$. Consider the dual of the tangent map $T\Phi\colon TQ_1\to TQ_2$ 
as a relation 
\begin{equation}\label{eq:R2}  \xymatrix{ T^\star Q_1 \ar@{-->}[r]^{T^\star\Phi}\ar[d] & T^\star Q_2 \ar[d] \\ Q_1\ar[r]_\Phi & Q_2
}\end{equation}
That is, $\mu_1\sim_{T^\star\Phi}\mu_2$ 
for $\mu_i\in T^\star_{m_i}Q_i$ means $m_2=\Phi(m_1)$ and 
$\mu_1=(T_{m_1}\Phi)^\star \mu_2$. 

\begin{lemma} \label{lem:27} 
Let $R\colon \AA_1\da \AA_2$ be a Courant morphism with base map 
$\Phi\colon Q_1\to Q_2$. Then 
\[ \a_2^\star \circ T^*\Phi=R\circ \a_1^\star.\]
That is, the dual of $\a=(\a_2,\a_1)$ restricts to a bundle map 
$\a^\star\colon \gr(T^\star \Phi)\to \gr(R)$. 
\end{lemma}
\begin{proof}
The assertion follows by dualizing the property $T\Phi\circ \a_2=\a_1\circ R$, using that $R^\star=R$
under the identification $\AA_i^\star=\AA_i$.  In detail, let $\mu_i\in T^\star_{m_i}Q_i$ with 
$\mu_1\sim_{T^\star\Phi}\mu_2$. For all $x_i\in \AA_{m_i}$ with   $x_1\sim_R x_2$, we have that
\[ \l  \a_1^\star(\mu_1),\ x_1\r
=\l \mu_2,\ T\Phi(\a_1(x_1))\r=\l\mu_2,\ \a_2(x_2)\r=\l \a_2^\star(\mu_2),\ x_2\r,\]
that is,  $\l (\a_2^\star(\mu_2),\ \a_1^\star(\mu_1)),\ (x_2,x_1)\r=0$. This shows 
$(\a_2^\star(\mu_2),\ \a_1^\star(\mu_1))\in \gr(R)^\perp=\gr(R)$, as desired. 
\end{proof}

Let $(\AA_i,E_i),\ i=1,2$ be Dirac structures on $Q_i$. We say that \eqref{eq:R} defines a
\emph{Dirac morphism} (or \emph{morphism of Manin pairs}) \cite{bur:cou}
\[ R\colon (\AA_1,E_1)\da (\AA_2,E_2)\]
if for all $m\in Q$, every $x_2\in (E_2)_{\Phi(m)}$ is $R$-related to a unique element 
$x_1\in (E_1)_m$. Equivalently, 
\[ \Phi^*E_2=R\circ E_1\]
where the composition is weakly transverse (when the composition is transverse, the Dirac morphism is called \emph{strong}). The resulting bundle map $\Phi^*E_2\to E_1$ defines a \emph{comorphism} of Lie algebroids $R\colon E_1\da E_2$: It is compatible with the anchor, and the map on sections 
$\Phi^*\colon \Gamma(E_2)\to \Gamma(E_1)$ preserves Lie brackets. 
%
%


\begin{definition}[\cite{bur:cou}] A \emph{Hamiltonian space} for a Dirac structure 
  $(\AA,E)$ on $Q$ is a manifold $M$ with a Dirac morphism
  \[ R\colon (\T M,T M)\da (\AA,E).\]
The base map $\Phi\colon M\to Q$ is called the \emph{moment  map}. 
\end{definition}
Given a Hamiltonian space, the resulting Lie algebroid comorphism $TM\da E$ defines
an \emph{action} of the Lie algebroid $E$ on the manifold $M$ \cite{bur:cou}. 
In particular, if $E$ is the action Lie algebroid for a $\g$-action on $Q$, then one
obtains a $\g$-action on $M$.

\begin{example}
Let  $Q$ be a manifold with a weak Poisson structure $(\T Q,E)$, thus 
 $E\cap TQ=0$, and let $M$ be a Hamiltonian space, defined by a Dirac morphism $R\colon (\T M,TM)\da (\T Q,E)$. By Proposition \ref{prop:forback},  the backward  image $F=TQ\circ R$ is a Lagrangian subbundle with  $TM\cap F=0$. We conclude that $F$ is again a weak Poisson structure. The map 
 $\Phi$ is anti-Poisson for these Poisson structures. 
\end{example}

\subsection{Exact Courant algebroids}
A Courant algebroid $\AA$ with base $Q$ is called \emph{exact} \cite{sev:let} if the 
following sequence is exact:
\[
\xymatrix@!@C=1.2em{0\ar[r]& T^\star Q\ar[r]^-{\a^{\star}} &  \AA\ar[r]^-{\a} & TQ\ar[r] & 0}.
\]
Equivalently, $\a^\star$ embeds $T^\star Q$ as a Lagrangian subbundle, defining a Dirac structure 
$(\AA,\ran(\a^\star))$.  
Using the Hilbert structure, the Lagrangian subbundle $\ran(\a^\star)$ admits a closed 
complement, and by Proposition \ref{prop:linred} one can choose this
complement to be Lagrangian. This determines a splitting
\[ j\colon TQ\to \AA\] such that $j(TQ)$ is a Lagrangian complement to
$\a^\star(T^\star Q)$.  We will refer to $j$ as an \emph{isotropic
  splitting}.  As observed by \v{S}evera \cite{sev:let}, the choice of
an isotropic splitting identifies $\AA \cong \T
Q_\eta$, where the closed 3-form $\eta\in\Omega^3(Q)$ is given by the formula
\begin{equation}\label{eq:etadef}
  \iota(v_1)\iota(v_2)\iota(v_3)\eta=\l j(v_1),\Cour{j(v_2),j(v_3)}\r.\end{equation}
The set of isotropic splittings is an affine space modeled on 2-forms:
Given $\varpi\in \Omega^2(Q)$, one obtains a new isotropic
splitting by the translation
\begin{equation}\label{eq:jprime}
  j'(v)=j(v)+\a^\star (\iota_v\varpi),\end{equation}
with the corresponding 3-form $\eta'=\eta+\d\varpi$.

A Courant morphism $R\colon \AA_1\da \AA_2$ between exact Courant algebroids will be called \emph{exact} if 
the sequence 
\[ 
\xymatrix@!@C=1.2em{0\ar[r]& \gr(T^\star \Phi)\ar[r]^-{\a^{\star}} &  \gr(R) \ar[r]^-{\a} & \gr(T \Phi)\ar[r] & 0}.
\]
is exact, where $\a=(\a_2,\a_1)$. It turns out that it is enough to know exactness at $\gr(T\Phi)$:
\begin{proposition}\label{prop:alleq}
The following conditions are equivalent:
\begin{enumerate}
\item $R$ is exact.
\item $R$ is full \cite{lib:qua}, that is, $\a|_{\gr(R)}\colon \gr(R)\to \gr(T\Phi)$ is surjective.
\item $\ran^\star(R)+\ran(\a_1^\star)=\AA_1$.
\end{enumerate}
Furthermore, in this case $R$ defines a strong Dirac morphism 
\begin{equation}\label{eq:dimo}
 R\colon (\AA_1, \ran(\a_1^\star))\da (\AA_2, \ran(\a_2^\star)).
 \end{equation}
\end{proposition}
\begin{proof} The implication (a) $\Rightarrow$ (b) is trivial. Suppose condition (b) holds. Then the map $\a^*|_{\gr(T^\star \Phi)}\colon \gr(T^\star \Phi)\to \gr(R)$ is injective, with image 
a closed subbundle of $\ker(\a|_{\gr(R)})\subseteq \gr(R)$. Since the $\AA_i$ are exact Courant algebroids, any $x\in \ker(\a|_{\gr(R)})$ is of the form $x=\a^*\mu$ for some 
$\mu\in T^\star(Q_2\times Q_1)$. For all $y\in \gr(R)$, we have $0=\l \a^*\mu,y\r=\l \mu,\a(y)\r$. Using again 
that $\a|_{\gr(R)}\colon \gr(R)\to \gr(T\Phi)$ is surjective, it follows that $\mu\in \gr(T^\star\Phi)$. 
This proves exactness at $\gr(R)$, and hence (a). On the other hand, 
condition (b) amounts to the statement that 
$\a_1(\ran^\star(R))=TQ_1$. Since $\a_1$ is the projection along $\ran(\a_1^\star)$, this is 
equivalent to (c). The final statement follows from Lemma \ref{lem:27}. 
\end{proof}
As a consequence of the fact that \eqref{eq:dimo} is strongly Dirac,  exact Courant morphisms can always be composed (see Proposition \ref{prop:automatic}). 
Another consequence is that one can `pull back' isotropic splittings:
\begin{proposition}\label{prop:pullsplit}
Let $R\colon \AA_1\da \AA_2$ be an exact Courant morphism, and $j_2\colon TQ_2\to \AA_2$ an isotropic splitting, with corresponding 3-form $\eta_2$. Then there is a unique isotropic splitting $j_1\colon TQ_1\to \AA_1$ such that 
\[ R\circ j_1=j_2\circ T\Phi.\]
The corresponding 3-form is $\eta_1=\Phi^*\eta_2$. 
\end{proposition}
\begin{proof}
The subbundle $F_2=j_2(TQ_2)$ is a Lagrangian complement to $\ran(\a_2^\star)$. Since \eqref{eq:dimo} is strongly Dirac, Proposition \ref{prop:forback} shows that its 
backward image $F_1=\Phi^*F_2\circ R$ is a Lagrangian complement to $\ran(\a_1^\star)$. 
Hence it is of the form $F_1=j_1(TQ_1)$ for an isotropic splitting $j_1$. By construction, this splitting 
satisfies $R\circ j_1=j_2\circ T\Phi$. Uniqueness of the isotropic splitting $j_1$ with this property 
follows from $\ker(\a_1)\cap \ker(R)=0$. 
Let $\eta_1$ be the corresponding 3-form. Then $\eta=\pr_2^*\eta_2-\pr_1^*\eta_1\in \Omega^3(Q_2\times Q_1)$ is the 3-form for the splitting $j=j_2\times j_1$ of $\AA_2\times \ol{\AA}_1$. If $v,v',v''$ are vector fields on $Q=Q_2\times Q_1$ that are tangent to $\gr(\Phi)$, then $j(v),j(v'),j(v'')$ restrict to sections of $\gr(R)$, 
and so does $\Cour{j(v'),j(v'')}$.  It follows that $\eta(v,v',v'')=\l j(v),\Cour{j(v'),j(v'')}\r$
vanishes along $\gr(\Phi)$, which is to say $\eta_1=\Phi^*\eta_2$. 
\end{proof}
\begin{proposition}\label{prop:exact1}
Let $\AA_1,\AA_2$ be exact Courant algebroids over $Q_1,Q_2$, with isotropic splittings $j_1,j_2$ identifying $\AA_i=\T Q_{i,\eta_i}$. Then an exact Courant morphism $R\colon \AA_1\da \AA_2$
with base map $\Phi\colon Q_1\to Q_2$ 
 is equivalently described by a 2-form $\om\in\Omega^2(Q_1)$ satisfying
\begin{equation}\label{eq:domega}
  \d\om=\eta_1-\Phi^*\eta_2. 
\end{equation}
\end{proposition}
This is standard in the finite-dimensional case (see e.g. \cite{gua:ge1}), and the proof carries over to infinite dimensions.  In one direction, the 2-form $\om$ relates the  splitting $j_1$ to the pullback 
of the splitting $j_2$ (see Proposition \ref{prop:pullsplit}). In the other direction, $\Phi$ and $\om$
determine an exact Courant morphism 
\[ \T \Phi_\om\colon \T Q_{1,\eta_1}\da \T Q_{2,\eta_2}\] 
by the condition 
\begin{equation}
\label{eq:morr}v_1+\mu_1\sim_{\T \Phi_\om} v_2+\mu_2\ \Leftrightarrow \ v_2=\Phi_*v_1,\ \ 
\mu_1=\Phi^*\mu_2+\iota(v_1)\om.
\end{equation} 
Under composition, $\T \Phi_\om\circ \T \Phi'_{\om'}=\T(\Phi\circ \Phi')_{\om+\Phi^*\om'}$. 

Given an \emph{exact Dirac structure}, i.e., a Dirac structure $(\AA,E)$ with $\AA$ exact, we define an \emph{exact Hamiltonian space} for $(\AA,E)$ to be a manifold $M$ together
with an  exact Dirac morphism
$R\colon (\T M, TM)\da (\AA,E)$.   
\begin{proposition}\label{prop:hamiltonian}
Let $(\AA,E)$ be an exact Dirac structure over $Q$, with a given isotropic splitting 
$j\colon TQ\to \AA$ identifying $\AA=\T Q_\eta$. Then the exact Hamiltonian spaces 
for $(\AA,E)$ are described by a map $\Phi\colon M\to Q$, a 2-form $\om\in \Omega^2(M)$, 
and a Lie algebroid action of $E$ along $\Phi$, satisfying 
\begin{enumerate}
\item \label{it:1a} $ \d\om=-\Phi^*\eta$, 
\item \label{it:1b} $\ker(\om)\cap \ker(T\Phi)=0$, 
\item \label{it:1c} $\iota(\sig_M)\om=-\Phi^*(j^\star \sig)$. 
\end{enumerate}
Here $\Phi^*\colon \Gamma(E)\to \Gamma(TM),\ \sigma\mapsto \sig_M$, is the Lie algebroid action, and $j^\star\colon \AA\to T^*Q$ is the bundle map dual to $j$.  
\end{proposition}
\begin{proof}
The exact Courant morphisms $R\colon \T M\da \T Q_\eta$ are of the form 
$R=\T \Phi_\om$ where $\om$ satisfies \eqref{it:1a}. 
Since $\ker(R)=\{v-\iota_v\om|\ v\in\ker T\Phi\}$, we 
see that the weak transversality condition $\ker(R)\cap TM$ for the composition 
$R\circ TM$ is equivalent to \eqref{it:1b}. The last property \eqref{it:1c} is equivalent to 
$R\circ TM=\Phi^*E$. 
\end{proof}

\begin{example} (See \cite{bur:cou}.)
Any leaf $i_\O\colon \O\hra Q$ of  an exact Dirac structure $(\AA,E)$ is naturally an exact Hamiltonian space. Here
\[ R\colon (\T\O,T\O)\da (\AA,E)\]
is uniquely defined by its properties that $v\sim_R x$ for $v\in T\O$ and $x\in E$ with 
$(T i_\O)(v)=\a(x)$, together with $\mu\sim_R \a^\star(\nu)$ for  $\mu\in T^\star \O$, 
$\nu\in T^\star Q$ such that $\mu\sim_{T^\star i_\O}\nu$. Given an isotropic splitting, identifying $\AA=\T Q_\eta$, we obtain a 2-form $\om\in\Omega^2(\O)$ with $\d\om=-i_\O^*\eta$. 
\end{example}

\begin{example}\label{ex:exex}
Let  $Q$ be a manifold with a weak Poisson structure $(\T Q,E)$, thus  $E\cap TQ=0$. Let $M$ be an exact Hamiltonian space, defined by an exact Dirac morphism $R=\T \Phi_\om\colon (\T M,TM)\da (\T Q,E)$. According to the proposition,  $\ker(\om)\cap \ker(T\Phi)=0$.  In fact, it is automatic 
that $\ker(\om)=0$. To see this, note that 
\[ \Phi^*E=R\circ TM,\ \ \gr(\om)=TQ\circ R.\]
Since $TQ\cap E=0$, Proposition \ref{prop:forback} shows that $\gr(\om)\cap TM=0$. Equivalently, $\ker(\om)=0$. 
\end{example}

\subsection{The Cartan-Dirac structure}
\label{sec:carcou}
Of special interest in this paper is the  \emph{Cartan-Dirac structure} on a Lie group $G$. 
We describe here its definition as an action Courant algebroid; later we will show that the 
same Dirac structure arises by reduction from the Lie-Poisson structure on the space of connections. 
\subsubsection{Definition of the Cartan-Dirac structure}
\label{ex:cardirac}
  Let $G$ be a Lie group. For $X\in\g$ we denote by $X^L,\ X^R$ the
  corresponding left, right-invariant vector fields. The Maurer-Cartan
  forms on $G$ will be denoted $\theta^L,\theta^R\in \Omega^1(G,\g)^G$;
  thus $\iota(X^L)\theta^L=X=\iota(X^R)\theta^R$.

  Suppose $G$ carries a bi-invariant pseudo-Riemannian metric, with corresponding $\Ad$-invariant metric $(X_0,X_{1})\mapsto X_{0}\cdot X_{1}$ on $\g$. We denote by $\ol{G}$ the Lie group $G$ with the opposite
  pseudo-Riemannian metric, and likewise by $\ol{\g}$ the Lie algebra
  $\g$ with the opposite metric. Let $D:=\ol{G}\times G$ act on $G$ by
  \[ (g_0,g_1).a=g_0\, a\, g_1^{-1}.\]
  The infinitesimal action $\dd=\ol{\g}\oplus\g\to \Gamma(TG)$ reads as
$ (X_0,X_1)\mapsto X_1^L-X_0^R$. It has co-isotropic stabilizers, hence  it defines an 
action Courant algebroid 
  \begin{equation}\label{eq:carcou} \AA=G\times \dd. \end{equation}
We refer to  $\AA$ as the \emph{Cartan-Courant algebroid}. If $\mf{s}\subseteq
  \dd$ is any subspace, the subbundle
  \[ E^{(\s)}=G\times {\mf{s}}\] is Lagrangian if and only if
  ${\mf{s}}$ is Lagrangian, and is involutive if and only if
  ${\mf{s}}$ is a Lie subalgebra. Thus, any Lagrangian Lie subalgebra
  $\mf{s}\subseteq \dd$ determines a Dirac structure.  The Dirac
  structure $E=E_{\g_\Delta}\subseteq \AA$ defined by the diagonal
  $\g_\Delta\subseteq \dd$ is called the \emph{Cartan-Dirac structure}.

  \begin{example}\label{ex:auto1}
    If $\kappa\colon\g\to \g$ is an orthogonal Lie algebra
    automorphism, then the graph
    $\gr(\kappa)=\{(\kappa(X),X)|\,X\in\g\}$ is a Lagrangian Lie
    subalgebra. Hence it determines a Dirac structure
    $E^{(\kappa)}=E_{\gr(\kappa)}$.  If the metric on $\g$ is
    positive definite, then any Lagrangian Lie subalgebra
    $\mf{s}\subseteq \dd$ arises in this way.  Indeed, any Lagrangian
    subspace is then given as the graph of an orthogonal
    transformation, and the condition that $\mf{s}$ is a Lie
    subalgebra means that this transformation preserves Lie brackets.
  \end{example}
%

\subsubsection{Splitting}
  The Cartan-Courant algebroid \eqref{eq:carcou} is exact, with an isotropic splitting $j\colon TG\to \AA$
   given at the group unit by the map 
  $\g\to \dd,\ X\mapsto \frac{1}{2} (-X,\, X)$. 
Equivalently, the map on sections $j\colon \Gamma(TG)\to \Gamma(\AA)$ is 
\begin{equation}\label{eq:courantsplitting}
 j(v)=\Big(-\hh \iota(v)\theta^R,\ \hh\iota(v)\theta^L\Big),
 \end{equation}
for $v\in \Gamma(TG)$. 
%
  By direct calculation, one find that the resulting 3-form 
  is the \emph{Cartan 3-form}
  \[ \eta=\f{1}{12}\theta^L\cdot [\theta^L,\theta^L],\]
  and that $\alpha=j^\star\circ \varrho\colon \dd\to \Omega^1(Q)$ is given by 
  \begin{equation}\label{eq:rhoG} \
 \alpha(X_0,X_1)=\hh(\theta^L\cdot X_1+\theta^R\cdot X_0),\ \ \
  (X_0,X_1)\in\dd.\end{equation}
Let $\varrho\colon \AA=G\times(\ol\g\oplus \g)\cong \T G_\eta$ be the resulting 
isomorphism. On the level of sections, 
\begin{equation}\label{eq:varrhocartan} \varrho(X_0,X_1)=X_1^L-X_0^R+
    \alpha(X_0,X_1)\end{equation}
for $(X_0,X_1)\in \ol\g\oplus\g$.  Taking $X_0=X_1=X$, we see that the Cartan Dirac
  structure is spanned by the sections
$X_G+\hh(\theta^L+\theta^R)\cdot X$ for $X\in\g$, where $X_G$ is the 
generating vector field for the conjugation action. 

\subsubsection{Hamiltonian spaces}
Suppose $\mf{s}\subseteq \dd$ is a Lagrangian Lie subalgebra, defining a Dirac structure $(\AA, E^{(\s)})$.  The data of Hamiltonian space $R\colon (\T M,TM)\da (\AA, E^{(\s)})$ for this Dirac
  structure gives, in particular, a Lie algebra action of $\s$ on $M$, such that $Y_M\sim_R \varrho(Y)$ for all  $Y\in\mf{s}$. 
  If $R$ is exact, one can use splittings to formulate these conditions in terms of differential forms. 
  Indeed, Proposition \ref{prop:hamiltonian} specializes to the following statement. 
  \begin{proposition}\label{prop:qh}
    An exact Hamiltonian space for the Dirac structure $(\AA, E^{(\s)})$ is
    equivalent to a triple $(M,\om,\Phi)$, consisting of a manifold $M$ 
    with an $\mf{s}$-action, a 2-form $\om\in\Omega^2(M)$ and an $\s$-equivariant map 
    $\Phi\colon M\to G$ satisfying
\begin{enumerate}
\item $ \d\om=-\Phi^*\eta$,
\item $  \ker(\om)\cap \ker(T\Phi)=0$, 
\item  $\iota(Y_M)\om=-\hh (Y_1\cdot \theta^L+Y_0\cdot \theta^R)$ for  all $Y=(Y_0,Y_1)\in\mf{s}$. 
\end{enumerate}     
%
  \end{proposition}
For the special case that $\mf{s}$ is the diagonal, we recover the axioms of a
  q-Hamiltonian $\g$-space as in
  \cite{al:mom}. 
%
  If the action of $\mf{s}$ integrates to an action of a Lie group $S$,
  and if $R$ is $S$-equivariant, we get a \emph{Hamiltonian
    $S$-space for $(\AA, E^{(\s)})$}: That is, $\om$ is $S$-invariant and 
    $\Phi$ is $S$-equivariant.     
    For instance, the $S$-orbits
  in $G$ are Hamiltonian $S$-spaces for $(\AA, E^{(\s)})$. Other examples are obtained by `fusion', as in \cite{al:mom}. 

%


\subsubsection{Multiplicative properites}
\label{subsec:mult}
\alejandro{Added this subsection back for future reference. It was written like this in a previous version (Dirac-holonomy.tex)}
  Give $\dd=\ol{\g}\oplus \g$ the \emph{pair groupoid} structure
  $\dd\rra\g$, with multiplication $(Y_0,Y_1)\circ
  (Y_0',Y_1')=(Y_0,Y_1')$ for $Y_1=Y_0'$. Taking the direct product of
  this groupoid with the group $G$, we obtain a groupoid $\AA\rra
  \g$. Since the groupoid multiplication covers the group
  multiplication of $G$, this is pictured as
  \[ \xymatrix{ {\AA} \ar@<2pt>[r]\ar@<-2pt>[r] \ar[d] &{\g}\ar[d]\\
    G \ar@<2pt>[r]\ar@<-2pt>[r] & \pt }\]
  Let $\Mult_\AA$ be the groupoid multiplication, defined on the
  subset of composable elements. Its graph $\on{gr}(\Mult_\AA)\subset
  \AA\times\ol{\AA}\times\ol{\AA}$ is a Dirac structure along the
  graph $\on{gr}(\Mult_G)$ of the group multiplication, defining a
  Courant morphism \cite{al:pur},
  \[ \on{gr}(\Mult_\AA)\colon \AA\times \AA\da \AA.\]
  Similarly, the groupoid inversion,
  $\on{Inv}_\AA=\on{Inv}_G\times\on{Inv}_\dd$ defines a Courant
  morphism
  \[ \on{gr}(\on{Inv}_\AA)\colon \AA\da \ol{\AA}.\]
  The Cartan Dirac structure $(\AA,E)$ makes $G$ into a \emph{Dirac
    Lie group}, in the sense that the groupoid multiplication defines
  a morphism of Manin pairs,
  \begin{equation}\label{eq:gfusion1}
    \on{gr}(\Mult_\AA)\colon (\AA,E)\times (\AA,E)\da (\AA,E),
  \end{equation}
  with underlying map the group multiplication \cite{al:pur, lib:dir}.
  More generally, suppose $\mf{s}_1,\ \mf{s}_2\subset \dd$ are
  Lagrangian Lie subalgebras, and that the groupoid multiplication
  $\mf{s}_1\circ \mf{s}_2$ is a transverse composition of linear relations.
  Then
  $\mf{s}_1\circ \mf{s}_2$ is a Lagrangian Lie subalgebra, and
  $\on{gr}(\Mult_\AA) $ defines a morphism of Manin pairs
  \begin{equation} \label{eq:gfusion} \on{gr}(\Mult_\AA)\colon
    (\AA,E_{\mf{s}_1})\times (\AA,E_{\mf{s}_1})\da
    (\AA,E_{\mf{s}_1\circ \mf{s}_2}).
  \end{equation}
  \alejandro{This last paragraphs is not used later. Remove?} In terms of the identification $\AA\cong \TG_\eta$ defined by the
  splitting of $\AA$, the multiplication morphism is given by the pair
  $(\Mult_G,\varsigma)$, where $\varsigma$ is a 2-form on $G\times G$
  satisfying
  \[ \d\varsigma=\Mult_G^*\eta-\pr_1^*\eta-\pr_2^*\eta,\]
  where $\pr_1,\pr_2\colon G\times G\to G$ are the two projections. As
  shown in \cite{al:pur}, this 2-form is
  \begin{equation}\label{eq:varsigma}
    \varsigma=-\hh \pr_1^*\theta^L\cdot \pr_2^*\theta^R.
  \end{equation}

\section{Reduction of Dirac structures}
\label{sec: reduction general}

In this section we continue to use the terms ``manifold'', ``vector bundle'', ``Lie
group'', etc. to refer to Hilbert manifolds, Hilbert vector bundles, Hilbert
Lie groups, and so on, unless otherwise specified.

\subsection{Actions on Courant algebroids}
Let $\AA$ be a Courant algebroid with base $Q$. 
A \emph{Courant derivation} of $\AA$ is a
linear operator $\ti{v}$ on $\Gamma(\AA)$, together with a vector
field $v$ on $Q$, satisfying
\[
\begin{split}
  v\l\sig_1,\sig_2\r&=\l \ti{v}\sig_1,\sig_2\r+\l \sig_1,\ti{v}\sig_2\r,\\ 
  \ti{v}\Cour{\sig_1,\sig_2}&=\Cour{\ti{v}\sig_1,\sig_2}+\Cour{\sig_1,\ti{v}\sig_2},\\
  \a(\ti{v}(\sig))&=[v,\a(\sig)].
\end{split}
\]
for all $\sig,\sig_1,\sig_2\in\Gamma(\AA)$.  These properties imply
the property, $\ti{v}(f\sigma)=v(f)\sigma+f \ti{v}(\sigma)$
for all $f\in C^\infty(Q)$ and $\sigma\in\Gamma(\AA)$.  Let
$\on{Der}(\AA)$ be the Lie algebra of Courant derivations of $\AA$.  
A Courant derivation is called  \emph{inner} if it is of the form $\ti{v}=\Cour{\sigma,\cdot}$
for some $\sig\in \Gamma(\AA)$; we refer to $\sig$ as a \emph{generator} of this Courant
derivation. 
Note that the map 
%
$\Gamma(\AA)\to \on{Der}(\AA),\ \ \sigma\mapsto \Cour{\sigma,\cdot}$
is bracket-preserving.

A \emph{Courant automorphism} of $\AA$ is a vector bundle automorphism
preserving the metric, the bracket, and compatible with the anchor.  One can informally regard $\on{Der}(\AA)$ as the Lie algebra of the group $\on{Aut}(\AA)$ of Courant automorphisms.

In particular, any 1-parameter group $g_t\in \Aut(\AA),\ t\in \R$ of Courant automorphisms determines a Courant derivation $\ti{v}(\sigma)=\f{\p}{\p t} \big|_{t=0}(g_{-t})^*\sigma$.

\begin{remark}
For \emph{exact} Courant algebroids, the group $\on{Aut}(\AA)$ and the Lie algebra $\on{Der}(\AA)$ are 
described in \cite[Section 2.1]{gua:ge1}. Choose a splitting to identify $\AA=\T Q_\eta$ for a closed 3-form $\eta$. Then $\on{Der}(\T Q_\eta)$ is the Lie subalgebra of the semidirect product $\Gamma(T Q)\ltimes \Omega^2(Q)$, consisting of pairs $(v,\varepsilon)$ with $\L_v\eta=\d \varepsilon$.  
The corresponding derivation $\ti{v}$ reads as 
\[ \ti{v}(w+\mu)=[v,w]+\L_v\mu-\iota_w\varepsilon,\]
for vector fields $w$ and 1-forms $\mu$. The inner derivation $\Cour{\sigma,\cdot}$ 
defined by $\sigma=w+\mu$ corresponds to the pair $(v,\varepsilon)$ with $v=w$ and $\varepsilon= -(\d \mu+\iota_w\eta)$. (We see that that $\Cour{\sigma,\cdot}=0$ if and only if $\sigma=\a^\star\mu$ with $\d\mu=0$; this description does not depend on the choice of splitting.)
Similarly
$\on{Aut}(\AA)$ is isomorphic to the subgroup of the semidirect product $\on{Diff}(Q)\ltimes \Omega^2(Q)$ consisting of 
pairs $(\Phi,\varepsilon)$ with $\Phi^*\eta+d\varepsilon=0$; the action of such a pair is given by 
the Courant morphism $\T\Phi_\varepsilon$. (Since $\Phi$ is a diffeomorphism, this morphism is given by an actual vector bundle automorphism of $\AA$.) 
\end{remark}
%
%
\begin{definition}
\begin{enumerate}
\item
Let $\g$ be a Lie algebra. 
A \emph{$\g$-action on a Courant algebroid} $\AA$ is a Lie algebra homomorphism
$\g\to \on{Der}(\AA), \ \xi\mapsto \xi_\AA$. A bracket preserving map
$\varrho\colon \g\to \Gamma(\AA)$
is said to define
\emph{generators} for the $\g$-action if
$\xi_\AA=\Cour{\varrho(\xi),\cdot}$ for all $\xi\in\g$.  
\item 
Let $G$ be a Lie group, acting on $\AA$ by Courant automorphisms. A $G$-equivariant map 
$\varrho\colon \g\to \Gamma(\AA)$
is said to define \emph{generators} for the $G$-action if 
\[ \xi_\AA:= \left.\tfrac{\p}{\p t}\right|_{t=0}\exp(-t\xi)^*=\Cour{\varrho(\xi),\cdot}\]
 for all $\xi\in\g$.  
\end{enumerate}
\end{definition}
Observe that $\a(\varrho(\xi))=\xi_Q$ are the generating vector fields for an action on $Q$.
We will use the same letter $\varrho$ to denote the associated bundle map
\[ \varrho\colon Q\times\g\to\AA,\ (m,\xi)\mapsto \varrho(\xi)_m.\] 
The dual map $\AA\to \g^\star$ is sometimes referred to as a \emph{moment map} for the Courant $G$ action. 
The set of generators for a Courant $G$-action is either empty, or is an affine space modeled on the space of $G$-equivariant maps from $\g$ into the kernel of  the map $\sigma\mapsto \Cour{\sigma,\cdot}$. For an exact Courant algebroid, this kernel is identified with the space of closed 1-forms.
\begin{example}
For any action Courant algebroid  $\AA=Q\times \dd$, the Lie algebra $\dd$ acts on $\AA$ by derivations, with the constant sections 
$\varrho\colon \dd\to \Gamma(\AA)$ as generators. In particular, the Cartan-Courant algebroid
$G\times (\ol{\g}\oplus \g)\cong \T G_\eta$ is $D=\ol{G}\times G$-equivariant, with the map 
\eqref{eq:varrhocartan} as generators.  
\end{example}
\begin{example}[Lie-Poisson structure]\label{ex:lieposson}
  Let $G$ be a Lie group  with Lie algebra $\g$,
  let $\xi_{\g^\star }\in\Gamma(T\g^\star )$ be the
  generating vector fields for the coadjoint action, and denote by
  $\d\mu\in \Omega^1(\g^\star,\g^\star)$ the tautological
  1-form. 
   Then the map $\varrho\colon \g\to \Gamma(\T\g^\star)$, 
  \begin{equation}
    \label{eq:coadj}\varrho(\xi)=\xi_{\g^\star }+\l\d\mu,\xi\r
  \end{equation}
defines isotropic generators for the $G$-action on $\T \g^\star$.
The Dirac structure $E\subseteq \T \g^\star $ spanned by the sections
$\varrho(\xi),\ \xi\in\g$ is a Poisson structure, in the strong sense that $\T \g^\star =E\oplus T\g^\star$. It is known as the  \emph{Lie-Poisson structure} on $\g^\star$.  
\end{example}

\subsection{Reduction of Dirac structures}
\label{subsec:coured}
Let $\AA\to Q$ be a $G$-equivariant Courant algebroid with generators $\varrho\colon Q\times \g\to \AA$.  We assume that the action is principal, i.e.~that $Q$ is
a principal bundle with base manifold $Q/G$. 
\begin{lemma} \label{lem:closlem}  The generators span a closed subbundle 
$\ran(\varrho)=\varrho(Q\times \g)\subset \AA$.
\end{lemma}
\begin{proof}
  The map $\varrho$ is a continuous bundle map. Since the
  composition $\a\circ \varrho\colon Q\times\g\to TQ$ is injective,
  with closed image, $\varrho$ must also have a closed image.
\end{proof}
We will describe the reduction procedure for the case that the generators $\varrho(\xi)$ are isotropic; 
equivalently, $\on{ran}(\varrho)=\varrho(Q\times \g)$ is isotropic. 

\begin{theorem}[
  \cite{bur:red}] \label{th:bcg} Suppose the generators are isotropic; 
  thus   
  $C=\varrho(Q\times\g)^\perp$ is coisotropic.  Then $\AA_C =
  C/C^\bot$ is a $G$-equivariant bundle, and the quotient bundle
  \[ \AA_\red=\AA_C/G\]
  with the induced fiber metric, bracket and anchor map is a Courant
  algebroid.  
  If $E\subseteq \AA$ is a $G$-invariant Dirac structure and $E+C$ is a
  closed subbundle, then the reduction $E_C = (E\cap C)/(E\cap C^\perp)$ is a
  $G$-equivariant bundle, and the reduced bundle
  \[ E_\red=E_C/ G \subseteq \AA_\red\]
defines a Dirac structure $(\AA_\red,\ E_\red)$ over $Q/G$.  
\end{theorem}
This result was proved in \cite{bur:red} in the finite-dimensional
setting, and for the case of exact Courant algebroids. However, the
proof immediately carries over to the general case.  A key observation
 is that the space $\Gamma(C)^G$ is closed under
Courant bracket, containing $\Gamma(C^\perp)^G$ as a Courant
ideal. Hence the Courant bracket descends to
$\Gamma(\AA_\red)=\Gamma(C)^G/\Gamma(C^\perp)^G$. Further, since
$\a(C^\perp)\subseteq TQ$ lies in the $G$-orbit directions, we obtain
a reduced anchor map $\a_\red\colon {\AA}_\red\to T(Q/G)$.
\begin{remarks}\mbox{}\label{rmks:redugral}
  \begin{enumerate}
  \item The condition that $E+C$ be closed is trivially satisfied if
    $E\subseteq C$, i.e. $\ran(\varrho)\subseteq E$.     
    One then has $E_\red=(E/\ran(\varrho))/G$, and 
    $\a(E_\red)\subseteq T(Q/G)$ is the image of $\a(E)\subseteq
    TQ$ under the quotient.
  \item \label{rmk:normal} 
Suppose that the action of $G$ on $\AA$ extends to an action of a  Lie group $U\supseteq G$,
with $U$-equivariant generators $\varrho\colon \u\to \Gamma(\AA)$ extending those of $\g$. 
We assume that $G$ is a normal subgroup of $U$, and that $\l\varrho(\xi),\varrho(\zeta)\r=0$ for all $\xi\in\g,\ \zeta\in\u$. Then the $G$-reduced Courant algebroid
    $\AA_\red$ inherits an action of the quotient group
    $U/G$, with generators $\varrho_\red\colon \mf{u}/\g\to
    \Gamma(\AA_\red)$ induced from those on $\AA$.
  \item There is a natural Courant morphism $q\colon \AA\da 
    \AA_\red$,  where $x\sim_{q} y$ if and only if $x\in C$, with $y$ its image in $\AA_\red$. 
    If the $G$-invariant Dirac structure $E$ has the property $E+C=\AA$, then $q$ defines a strong Dirac morphism $q\colon (\AA,E)\da (\AA_\red,E_\red)$. In the opposite extreme, if 
    $\ran(\varrho)\subseteq E$, this is a Dirac \emph{co}morphism. 
\item The closed subbundle $\a^{-1}(\ran(\varrho_Q))\subseteq \AA$ decomposes as a direct sum 
$\ker(\a)\oplus C^\perp$; hence $\ker(\a)^\perp+C$ is again closed. 
If $\AA$ is exact, so that $\ker(\a)=\ran(\a^\star)$ is Lagrangian, this shows that 
$\ran(\a^\star)+C=\AA$, or equivalently $\a(C)=TQ$. Using these facts, we see that $\AA_\red$ is exact, 
 and the Courant morphism $q$ is exact as well. Furthermore, 
\[ q\colon (\AA, \,\ran(\a^\star))\da (\AA_\red,\,\ran(\a_\red^\star)).\]
is a strong Dirac morphism.
  \end{enumerate}
\end{remarks}


\subsection{Reduction of Dirac morphisms}
\label{subsec:equiv}
\alejandro{This subsection has changes, as it now considers two groups $G_1\neq G_2$. I took most of it from a previous version of the draft that incorporated this. (I had to revise some references to old-setted Thm \ref{th:Rred} across the text as well.)} $\AA_i\to Q_i$, $i=1,2$, be $G_i$-equivariant Courant algebroids over $Q_i$, with generators $\varrho_i\colon \g_i\to
\Gamma(\AA_i)$. A Courant morphism $R\colon \AA_1\da \AA_2$, with base map $\Phi:Q_1 \to Q_2$, is called
\emph{equivariant} with respect to a group homomorphism $f\colon
G_1\to G_2$ if
\[ x_1\sim_R x_2 \Rightarrow g\cdot x_1\sim_R f(g)\cdot x_2\]
for all $g\in G_1$ and $x_i\in \AA_i$. It \emph{intertwines the generators} if 
\begin{equation}\label{intertwgen} \varrho_1(\xi)\sim_R \varrho_2(f(\xi))\end{equation}
for all $\xi\in\g_1$. Let $\ran(\varrho_i)=\varrho_i(Q_i\times \g_i)\subseteq \AA_i$ denote the closed  subbundles spanned by the generators.

\begin{theorem}\label{th:Rred}
  Suppose in addition that  the generators for the $G_i$-actions are isotropic and that
  the $G_i$-actions on $Q_i$ are principal.  Then:
  \begin{enumerate}
  \item The Courant morphism $R$ descends to a Courant morphism
\[ \xymatrix{
 (\AA_1)_\red\ar@{-->}[r]_{R_\red}\ar[d] & (\AA_2)_\red\ar[d] \\
 Q_1/G_1\ar[r]_{\Phi_\red} & Q_2/G_2
 }\]  
   
   \item The morphism $R_\red$ has the property
   \[ q_2 \circ R = R_\red \circ q_1,\]
   where $q_i\colon \AA_i\da (\AA_i)_\red$ are the reduction morphisms. 
   
  \item Suppose the Courant algebroids $\AA_i$ are exact. Then, the reduction procedure gives a one-to-one
correspondence between Courant morphisms $(\AA_1)_\red\da (\AA_2)_\red
$, and equivariant Courant morphisms $\AA_1\da \AA_2$ intertwining the generators. 

 \end{enumerate}
\end{theorem}

\begin{proof}
\alejandro{The proof changed for $G_1\neq G_2$. There was a proof already typed for this case for items (a,b) in an older version of the draft. I took it and revised it. The proof of point (c) is completely new, but quite parallel to the one that was done for $G_1=G_2$.  Im quite sure it is ok, but it is good to double check this proof anyway.}
  Consider the $G_2\times G_1$-equivariant Courant algebroid $\AA = \AA_2\times\ol{\AA}_1$
  over $Q=Q_2\times Q_1$, with generators $\varrho =
  \varrho_2\times \varrho_1$. Let $\pi_i\colon Q_i\to Q_i/G_i$ be the quotient maps, write $\pi = \pi_2 \times \pi_1$ and $G= G_2 \times G_1$, and 
define  $C_i=\varrho_i(Q_i\times \g_i)^\perp$. The subbundle  $C = C_1\times C_2$ defines the reduction to $\AA_{red}= \AA_C/(G_2\times G_1) =(\AA_2)_{\on{red}}\times(\ol{\AA}_1)_{\on{red}}$.
  The graph
  $\on{gr}(\Phi)$ is invariant under the action of the diagonal subgroup $(G_1)_\Delta=\{(f(g),g)|g\in G_1\}\subset G$.  Let
  \[
  \wt{\on{gr}}(\Phi)=G\cdot\on{gr}(\Phi)=G\times_{(G_1)_\Delta}\on{gr}(\Phi)
  \]
  be the flow-out under the $G$-action. Then
  \[
  \wt{\on{gr}}(\Phi)/G=\on{gr}(\Phi)/(G_1)_\Delta=\on{gr}(\Phi_{\on{red}}).
  \]

  \begin{enumerate}
  \item By definition of a Dirac morphism, ${\gr(R)}$ is a Dirac structure along $\gr(\Phi)$: whenever two sections of $\AA$ restrict over $\gr(\Phi)$ to sections of ${\gr(R)}$, then so does their Courant bracket. Also, ${\gr(R)}$ is invariant under the action of $(G_1)_\Delta$ and its flow-out
    \[ 
    \wt{\on{gr}}(R)=G\cdot \on{gr}(R)\cong G\times_{(G_1)_\Delta} \on{gr}(R)\to
    \wt{\on{gr}}(\Phi)
    \]
    is a closed Lagrangian subbundle. Since its space of sections
    is generated by $\Gamma(\wt{\on{gr}}(R))^G\cong \Gamma(\on{gr}(R))^{(G_1)_\Delta}$,
    it is also involutive. Hence it is a Dirac structure along
    $\wt{\on{gr}}(\Phi)$.

    Along $\on{gr}(\Phi)$, the sum 
    \[
    {\gr(R)}+\ran(\varrho)|_{\gr(\Phi)}
    \]
    is a closed subbundle, for the following reason: since
    $\varrho_1(\xi)\sim_R \varrho_2(f(\xi))$ for all $\xi\in\g_1$, it
    coincides with the direct sum of closed subbundles
    $R\oplus \ran(\varrho_2)$, and this is mapped
    by the anchor to the closed subbundle $T\wt{\on{gr}}(\Phi) =
    T\on{gr}(\Phi)\oplus \ran(\a\circ \rho_2)$ in a way
    which preserves the direct sum decomposition and is injective on
    the second factor. As a result, its flow-out
    $\wt{\on{gr}}({R})+\ran(\varrho)|_{\wt{\on{gr}}(\Phi)}$ under the action
    of $G$ is also closed.  
    
It follows that $\wt{\on{gr}}({R})_C = (\wt{\on{gr}}(R)\cap C)/(\wt{\on{gr}}(R)\cap C^\perp)$ is a Lagrangian subbundle of $\AA_C$ along $\wt{\on{gr}}(\Phi)$ and hence that 
    \[\gr(R)_{\on{red}}= \wt{\on{gr}}(R)_C/G,\] is a Lagrangian subbundle of $\AA_{\on{red}} = \AA_C/G$ along the graph of
    $\Phi_{\on{red}}$.
To check integrability of ${\gr(R)}_\red$, it is enough to argue locally. Let $\sig,\sig'$ be sections of $\AA_\red$ defined near $(\pi_2(\Phi(m)),\pi_1(m))$, and restricting to sections of ${\gr(R)}_\red$. Using local triviality of the principal bundle, these lift to $G$-invariant  sections $\hat{\sigma},\hat{\sigma}'$ of $C\subseteq \AA$, defined near $(\Phi(m),m)$, and restricting to sections of ${\wt{\gr}(R)}$. The Courant bracket $\Cour{\hat{\sigma},\hat{\sigma}'}$ has the same property, by integrability of both $C$ and ${\wt{\gr}(R)}$. Hence $\Cour{\sigma,\sigma'}$ descends to a section of ${\gr(R)}_\red$.

  \item Let $x_1\in \AA_1$ and $y_2\in
    (\AA_2)_{\on{red}}$; we must show that
    $x_1\sim_R x_2\sim_{q_2} y_2$ for some $x_2\in \AA_2$
    if and only if $x_1\sim_{q_1}y_1\sim_{R_{\on{red}}}y_2$ for some
    $y_1\in (\AA_1)_{\on{red}}$. Given the latter
    property, since $y_1\sim_{R_{\on{red}}}y_2$ the definition of
    $R_{\on{red}}$ gives $\ti{x}_1\sim_R \ti{x}_2$ for some
    $\ti{x}_i\in \AA_i$ with $\ti{x}_i\sim_{q_i} y_i$. The difference
    $\ti{x}_1-x_1$ is $q_1$-related to $0$, hence it is of the form
    $\varrho_1(\xi_1)|_m$ for some $\xi_1\in\g_1$.  Put
    $x_2=\ti{x}_2-\varrho_2(f(\xi_1))$. Then $x_1\sim_R x_2\sim_{q_2}
    y_2$ as desired. Conversely, given $x_2$ with this property,
    so that $x_2\in C_2$, the condition $x_1\sim_R x_2$ implies that
    for all $\xi_1\in\g_1$, $\l x_1,\varrho_1(\xi_1)\r=\l
    x_2,\varrho_2(f(\xi_1))\r=0$.  Hence $x_1\in C_1$, which determines an element 
    $y_1$ with $x_1\sim_{q_1} y_1$.  By definition of
    $R_{\on{red}}$, the property $x_1\sim_R x_2$ descends to
    $y_1\sim_{R_{\on{red}}} y_2$.

\item Given $R$, we show how to express $R$ in terms of $R_\red$.  
Let $p_i\colon C_i\to\AA_i$ be the quotient maps, and $p=p_2\times p_1$.  
The pre-image $p^{-1}(\gr(R_\red))$ is a Lagrangian subbundle along $\pi^{-1}(\gr(\Phi_\red))=
\wt{\gr}(\Phi)$. 
Its intersection with $\a^{-1}(\gr(T\Phi))$ is contained in $\gr(R)$. 
Since $\AA$ is an exact Courant algebroid, $D=\a^{-1}(\gr(T\Phi))$ is 
a closed coisotropic subbundle; the orthogonal bundle is $D^\perp=\a^\star(\gr(T^\star\Phi))$. 
Note that $p^{-1}(\gr(R_\red))|_{\gr(\Phi)}+D=\ran(\rho)|_{\gr(\Phi)}+D$ is closed, by the same reasoning as in part (a). 
Reducing $p^{-1}(\gr(R_\red))|_{\gr(\Phi)}$ with respect to $D$, and then taking the inverse image under the quotient map $D\to D/D^\perp$, we obtain a Lagrangian subbundle
\begin{equation}\label{eq:RR2}
 (p^{-1}(\gr(R_\red))|_{\gr(\Phi)}\cap D)+D^\perp
\end{equation}
along $\gr(\Phi)$. 
Since both summands lie in $\gr(R)$, the sum \eqref{eq:RR2} is in fact equal to $\gr(R)$.  
Conversely, if the Courant morphism $R_\red$ is given, we can take \eqref{eq:RR2} as the definition of $R$. 
This $R$ is $(G_1)_\Delta$-equivariant and intertwines the generators, and an argument similar to (a) shows that it is integrable.  
\qedhere
%

  \end{enumerate}
\end{proof}

\alejandro{Added this sentence and the corresponding first sentence in the next Prop. From this point on, I made are no other changes in this subsection.}For the rest of this Section we shall focus on the case in which there a single group $G$ acting on both $\AA_i$ and equivariance holds with respect to the identity map.

\begin{proposition}\label{prop:Diracred}
Consider the setting of Theorem \ref{th:Rred} with $G_1=G_2=G$ and $f=id_G$.  Suppose $E_i\subseteq \AA_i$ are $G$-invariant Dirac
structures such $E_i+\ran(\varrho_i)$ are closed
    subbundles,  and that  $R$ defines a Dirac morphism $R\colon (\AA_1,E_1)\da (\AA_2,E_2)$.  
    Suppose also that for all $m\in Q_1,\ \xi\in\g$, 
\begin{equation}\label{eq:technical}
\varrho_2(\xi)_{\Phi(m)}\in E_2 \Rightarrow \varrho_1(\xi)_m\in E_1.
\end{equation}
    Then $R_\red$ defines a Dirac morphism 
    \[ R_\red\colon ((\AA_1)_\red,(E_1)_\red)\da
    ((\AA_2)_\red,(E_2)_\red).\]
If the composition $R\circ E_1$ is transverse, then so is the composition $R_\red\circ (E_1)_\red$. 
\end{proposition}
\begin{proof}
  We have to show that every $y_2\in \Phi_\red^*(E_2)_{\red}$ is $R_\red$-related to  
 a unique element $y_1\in (E_1)_\red$. Let $x_2\in E_2\cap C_2$ be a lift of $y_2$. 
 Since $R$ is a Dirac morphism, there exists $x_1\in E_1$ with $x_1\sim_R x_2$. 
 This element satisfies $\l x_1,\varrho_1(\xi)\r=\l x_2,\varrho_2(\xi)\r=0$ for all $\xi\in \g$, hence 
 $x_1\in C_1$. Letting $y_1\in (E_1)_\red$ be the image, we get $y_1\sim_{R_\red} y_2$. 
 For uniqueness, suppose $y_1\in (E_1)_\red$ satisfies 
    $y_1\sim_{R_\red} 0$. Choose elements $x_i\in C_i\cap E_i$ with 
    $x_1\sim_R x_2$, $x_1\sim_{q_1} y_1$ and $x_2\sim_{q_2} 0$. The last condition gives
$x_2=\varrho_2(\xi)_{\Phi(m)}$ for some $\xi\in\g$.   Then $x_1\sim_R  x_2$ but also 
$\varrho_1(\xi)\sim_R x_2$. By assumption \eqref{eq:technical}, and since $R$ is a Dirac morphism, this implies $x_1=\varrho_1(\xi)$. Hence  $y_1=0$. 

Suppose that the composition of $R$ with $E_1$ is transverse, that is, $\ran^\star(R)+E_1=\AA_1$. 
We want to prove $\ran^\star(R_\red)+(E_1)_\red=(\AA_1)_\red$. Given $v_1\in    (\AA_1)_\red$, 
let  $u_1\in C_1$ be a preimage. Write $u_1=x_1+a_1$ with   $x_1\in E_1$ and 
$a_1\in \ran^\star(R)$. Then  $a_1\sim_R a_2$ for some $a_2\in   \AA_2$. By assumption, for  all $w_2\in E_2\cap \ran(\varrho_2)$ there exists $w_1\in E_1\cap \ran(\varrho_1)$ with $w_1\sim_R w_2$. Therefore, 
\[ \l a_2,w_2\r=\l a_1,w_1\r=0\]
for all $w_2\in E_2\cap \ran(\varrho_2)$, which proves $a_2\in E_2+C_2$. Modifying the element $x_1$, we may   arrange that the $E_2$-component of $a_2$ is zero. Hence $a_2\in   C_2$ descends to an element $b_2\in (\AA_2)_\red$.  Using part (b) from Theorem \ref{th:Rred}, the   property $a_1\sim_R a_2\sim_{q_2} b_2$ 
shows the existence of an element $b_1$ with $a_1\sim_{q_1} b_1\sim_{R_\red} b_2$. 
In particular, $b_1\in \ran^\star(R_\red)$. It also follows that $a_1\in C_1$, and hence 
$x_1=u_1-a_1\in E_1\cap C_1$. Letting $y_1\in (E_1)_\red$ be its  image, we obtain $v_1=y_1+b_1$.  
\end{proof}

\begin{remark}\label{rem:Rred}
Condition \eqref{eq:technical} is automatic in the following two cases 
\begin{enumerate}
\item $\ran(\varrho_2)\cap E_2=0$, 
\item $\ran(\varrho_1)\subseteq E_1$. 
\end{enumerate}
As a special case, suppose $R\colon (\T M,TM)\da (\AA,E)$ is a
$G$-equivariant Hamiltonian space, with base map $\Phi\colon
M\to Q$, and where the $G$-action on $\T M$ is the standard
lift of a $G$-action on $M$. If the $G$-action on  $Q$ 
is a principal action, then by equivariance the action on $M$ is again
a principal action.  Hence we obtain a Hamiltonian space
$R_\red\colon (\T (M/G),T(M/G))\da
(\AA_\red,E_\red)$. 
\end{remark}

\subsection{Reduction of exact Courant algebroids}
\label{subsec:exact-red}
Suppose $\AA\to Q$ is a $G$-equivariant exact Courant algebroid, and let $j\colon
TQ\to \AA$ be a $G$-equivariant isotropic splitting, identifying $\AA\cong \T Q_\eta$ 
for a $G$-invariant closed 3-form $\eta\in \Omega^3(Q)$. The following result describes isotropic generators 
for the action in terms of the splitting. Recall that the Cartan complex of equivariant  differential forms on $Q$ is the space of $G$-equivariant polynomial maps $\beta\colon \g\to \Omega(Q)$, with the 
equivariant differential 
\[ (\d_G\beta)(\xi)=\d\beta(\xi)-\iota(\xi_Q)\beta(\xi).\]
\begin{proposition}\label{prop:burdi}
 \cite{bur:red}  
For $\AA\cong \T Q_\eta$ as above, there is a 1-1 correspondence between $G$-equivariant isotropic generators $\varrho\colon \g\to \Gamma(\AA)$ and closed equivariant extensions 
\begin{equation}\label{eq:eqetaform}
\eta_G(\xi)=\eta+\alpha(\xi)
\end{equation}
of the 3-form $\eta$. That is, $\alpha\colon \g\to \Omega^1(Q)$ is a $G$-equivariant map
such that $\d_G\eta_G=0$.  Under this correspondence, 
\begin{equation}\label{eq:alpha}
  \varrho(\xi)=j(\xi_Q)+\a^\star (\alpha(\xi)).
\end{equation}
 Changing the splitting by an invariant 2-form $\varpi\in \Omega^2(Q)^G$
  modifies $\eta_G$ to $\eta_G'=\eta_G+\d_G\varpi$. 
\end{proposition}
\begin{remark}
If the generators are not necessarily isotropic, one finds instead that $\d_G\eta_G(\xi)=\hh \l\varrho(\xi),\varrho(\xi)\r$. 
\end{remark}

We now make the additional assumption that the $G$-action
on $Q$ is principal, as in Theorem \ref{th:bcg}, with quotient map
$\pi\colon Q\to Q/G$. Suppose isotropic generators $\varrho\colon \g\to \Gamma(\AA)$ are given. 
\begin{definition}
  An isotropic splitting $j\colon TQ\to \AA$ is called
  \emph{$\g$-horizontal} if $\varrho(Q\times \g)\subseteq j(TQ)$, or
  equivalently $ \varrho(\xi)=j(\xi_Q)$ for all $\xi\in\g$. It is
  called \emph{$G$-basic} if it is both $G$-invariant and
  $\g$-horizontal.
\end{definition}
Thus, an invariant isotropic splitting is $G$-basic if and only if $\alpha=0$. 
There is a 1-1 correspondence between $G$-basic splittings
$j$ of $\AA$ and isotropic splittings $j_\red$ of
$\AA_\red$.  Under this correspondence, $j(TQ)$ is the pre-image
of $j_\red(T(Q/G))$ under the quotient map. The three-form of a 
$G$-basic splitting $j$ coincides with its equivariant extension, 
and equals the pullback of the
three-form of the reduced splitting $j_\red$:
\begin{equation*}
  \eta = \pi^*\eta_\red.
\end{equation*}
\begin{proposition}\label{prop:basplit}
  Let $Q\to Q/G$ be a principal $G$-bundle with connection
  $\theta\in\Omega^1(Q,\g)^G$.  Let $\AA\to Q$ be a $G$-equivariant
  Courant algebroid with isotropic generators, and let $j\colon TQ\to
  \AA$ be a $G$-invariant isotropic splitting.  
   Put
  \begin{equation}\label{eq:varpi} 
  \varpi=-\alpha(\theta)+\hh c(\theta,\theta)\in \Omega^2(Q)^G,
  \end{equation}
  where $\alpha$ is given by \eqref{eq:alpha}, and $c(\xi,\xi')=\iota(\xi_Q)\alpha(\xi')\in C^\infty(Q)$.  Twisting the
  splitting $j$ by $\varpi$, we obtain a $G$-basic splitting. The resulting 3-form on $Q/G$ satisfies 
 \begin{equation}\label{eq:neweta}
 \pi^*\eta_\red=\eta+\d\varpi.\end{equation}
\end{proposition}
\begin{proof}
  The $\varpi$-twisted splitting $j'$ is given by \eqref{eq:jprime},
  and the corresponding 1-forms $\alpha'(\xi)$ are
  $\alpha'(\xi)=\alpha(\xi)-\iota(\xi_Q)\varpi$. But
  \[ \iota(\xi_Q)\varpi
  =-(\iota(\xi_Q)\alpha)(\theta)+\alpha(\xi)+c(\xi,\theta)=\alpha(\xi).\]
  Thus $\alpha'(\xi)=0$.
\end{proof}
\begin{remark}
Note that if the splitting $j$ was $G$-basic to begin with, then $\varpi=0$, and hence $j'=j$, for any choice of connection $\theta$.
\end{remark}

\begin{remark}
The 2-form $\varpi$ also appears in the context of lifting the structure group of the principal bundle to a central extension by $\U(1)$. See Appendix \ref{sec:brylinski}. 
\end{remark}

The reduction of an exact Courant morphism is again exact:
\begin{proposition}\label{prop:Diracred1}
In the setting of Theorem \ref{th:Rred}  with $G_1=G_2=G$ and $f=id_G$,  suppose that Courant algebroids $\AA_i$ and the 
Courant morphism $R$ are exact. Then so are $(\AA_i)_\red$ and 
the Courant morphism $R_\red$.
\end{proposition}
\begin{proof}
In the exact case,  $R$ defines a strong Dirac morphism $(\AA_1,\ran(\a_1^\star))\da 
(\AA_2,\ran(\a_2^\star))$.  We have $\ran(\a_i^\star)_\red
=\ran((\a_i^\star)_\red)$, hence by Proposition \ref{prop:Diracred}
 applied to $E_i=\ran(\a_i^\star)$
we obtain a strong Dirac morphism 
$R_\red\colon ((\AA_1)_\red,(\ran(\a_1^\star))_\red)\da 
((\AA_2)_\red,(\ran(\a_2^\star))_\red)$. In turn, this means that $R_\red$ is exact. 
\end{proof}
We now describe the reduction of exact Courant morphisms in terms of isotropic splittings. Suppose $\AA_i\to Q_i$ for $i=1,2$ are $G$-equivariant exact Courant algebroids, with isotropic generators $\varrho_i\colon \g\to \Gamma(\AA_i)$. 
Let $j_i\colon TQ_i\to \AA_i$ be $G$-equivariant isotropic splittings, identifying $\AA_i=\T Q_{i,\eta_i}$ for  closed 3-forms $\eta_i$.

\begin{proposition}\label{prop:reductionofexactmorphisms} 
A  $G$-equivariant exact Courant morphism
\[ \T\Phi_\om \colon \T Q_{1,\eta_1}\da \T Q_{2,\eta_2}\]
intertwines the generators (cf.~ Equation \eqref{intertwgen} with $f=id_G$) if and only if 
\begin{equation}\label{eq:eq2form}
\d_G\om=\eta_{1,G}-\Phi^*\eta_{2,G}.\end{equation}
If the $G$-actions on $Q_i$ are principal actions, and $\theta_i\in\Omega^1(Q_i,\g)$ are connection 1-forms, defining 2-forms $\varpi_i\in \Omega^2(Q_i)$ as in \eqref{eq:varpi} and 3-forms $\eta_{i,\red}$ 
as in \eqref{eq:neweta}, then the reduced Courant morphism is 
\[ (\T \Phi_\red)_{\om_\red} \colon (\T Q_{1,\red})_{\eta_{1,\red}}\da (\T Q_{2,\red})_{\eta_{2,\red}}\]
where $\Phi_\red\colon Q_1/G\to Q_2/G$ is the map induced by $\Phi$, and $\om_\red$ is given by 
\[
 \pi_1^*\om_\red=\om+\varpi_1-\Phi^*\varpi_2.\]
\end{proposition}
\begin{proof}
By definition, the 2-form $\om$ relates the splitting $j_1$ with the `pullback' of the splitting $j_2$. 
Hence, \eqref{eq:eq2form} follows from 
Proposition \ref{prop:burdi}. Suppose now that the $G$-actions are principal. 
Given connection 1-forms $\theta_i$ and the corresponding 2-forms $\varpi_i$, let $j_i'$ be the 
$G$-basic splittings obtained by twisting $j_i$ by the 2-forms 
$\varpi_i$. The 3-forms change to $\eta_i'=\eta_i+\d\varpi_i$, which are $G$-basic and in particular coincide with their equivariant extensions: $\eta'_{i,G}=\eta'_i$. The
2-form describing $R = \T\Phi_{\omega}$ relative to the new splitting is
$\om'=\om+\varpi_1-\Phi^*\varpi_2$. Equation \eqref{eq:eq2form} gets replaced with 
$\d_G\om'=\eta'_{1}-\Phi^*\eta'_{2}$, which shows in particular that $\om'$ is $G$-basic. 
The resulting 2-form 
$\om_\red$ with $ \pi_1^*\om_\red=\om'$ describes the exact morphism $R_\red$. 
\end{proof}

\section{The Hilbert principal bundle of connections}
\label{sec:hol} 
Let $G$ be a connected finite-dimensional Lie group.  
The holonomy fibration is defined to be the space of connections $\A_{\II}$ on the trivial $G$-bundle over the interval $\II = [0,1]$.  By imposing appropriate regularity conditions, the space $\A_{\II}$ is a principal bundle for the Hilbert Lie group of gauge transformations which are trivial at the boundary $\p\II$.  The principal bundle projection is the map to $G$ given by the holonomy along the interval.  A slight modification of this holonomy fibration, which makes contact with the usual theory of loop groups, is studied in Section~\ref{sec:conns1}; there we consider connections on the trivial $G$-bundle over the circle instead of the interval.  Also, in our study of the geometry of these fibrations it will be useful to choose principal connections, which may be done via the Caloron correspondence, reviewed in Appendix~\ref{app:cal}.

%
%


\subsection{Sobolev space notation} 
\label{subsec: convent} 

We use the following basic properties of Sobolev spaces (see
e.g. \cite[Section 11]{bo:el} or \cite[Section 14]{at:mo}). Let $r\ge 0$. For a finite-dimensional compact manifold $M$, possibly
with boundary, let $\HH_r(M)$ denote the order $r$ Sobolev space of
functions. In particular, we have $\HH_0(M)=\LL\!^2(M)$.  The spaces
$\HH_r(M)$ are Hilbert spaces, with compact inclusions $\HH_s(M)\subseteq \HH_r(M)$ for $s>r$. 
By the Sobolev embedding
theorem, $\HH_r(M)\subseteq \mathsf{C}^l(M)$ for $r-\hh\dim(M)>l$.  The
space $\HH_r(M)$ is a Banach algebra for $r-\hh \dim M>0$, and for
$0\le j\le r$ the space $\HH_j(M)$ is a module over this Banach algebra. If
$Z\subseteq M$ is a submanifold, and $r-\hh \on{codim}(Z)>0$, then the
restriction of continuous functions from $M$ to $Z$ extends to a
continuous linear map $\HH_r(M)\to \HH_{r-\hh\on{codim}(Z)}(Z)$, with
a continuous right inverse.

If $N$ is another finite-dimensional manifold and $r-\hh \dim (M)>0$,
one defines spaces $\on{Map}_{\HH_r}(M,N)\subseteq C^0(M,N)$ of maps
$M\to N$ of Sobolev class $r$, by choosing local charts for $N$. In
particular, if $G$ is a finite-dimensional Lie group, and $r-\hh \dim
M>0$, then $\on{Map}_{\HH_r}(M,G)$ is a Hilbert Lie group under
pointwise multiplication. 

We denote by $\Omega^k_{\HH_r}(M)$ the sections of $\wedge^k T^\star M$ of Sobolev class $r$. 


\subsection{The holonomy map as a principal bundle projection}
\label{subsec: holonomy bundles}
We make use
of several elementary results in gauge theory (see e.g.~ \cite[Appendix A]{fre:ins}), specialized to the case of 1-dimensional manifolds. 
Fix a real number $r\ge 0$. The space of connections on the trivial $G$-bundle over the interval
$\II=[0,1]$ with Sobolev class $r$ is a Hilbert manifold 
\[
\A_\II = \Omega^1_{\HH_r}(\II,\g).
\]
Since $r\ge 0$, the
space of maps
\[ 
G_\II=\on{Map}_{\HH_{r+1}}(\II,G)
\]
defines a Hilbert Lie group, with Lie algebra $\g_\II=\Omega^0_{\HH_{r+1}}(\II,\g)$. This group acts smoothly 
by gauge transformations
\begin{equation}\label{eq:gaugeaction}
g\cdot A=\Ad_g(A)-g^*\theta^R,
\end{equation}
for $g\in G_\II$ and $A\in \A_\II$.   Here $\theta^R\in\Omega^1(G,\g)$ is the right-invariant Maurer-Cartan form on $G$.
Note that $g$ is taken to have Sobolev class $r+1$ because the involvement
of derivatives implies that $g^*\theta^R$ has class $r$. Given $\xi\in \g_\II$, the corresponding 
generating vector field $\xi_{\A_{\II}}$ is given by 
\begin{equation}\label{eq:fundvf}
  \xi_{\A_{\II}}|_A=\partial_A\xi, 
\end{equation}
where 
\begin{equation}\label{eq:covder}
\partial_A = \partial + \ad(A) \colon \Omega^0_{\HH_{r+1}}(\II,\g)\to\Omega^1_{\HH_{r}}(\II,\g)
\end{equation}
is the exterior covariant derivative associated to $A\in\A_\II$. 
The action \eqref{eq:gaugeaction}  is transitive: given $A\in\A_\II$, the equation
\[ A =g^{-1}\cdot 0=g^*\theta^L\] is a first order ordinary differential equation for
$g\in G_\II$, and so has a unique solution once an initial condition
$g(0)$ is chosen. Furthermore, this solution lies in $\HH_{r+1}$ by
standard elliptic theory, as required. We define the holonomy map
$\Hol\colon \A_{\II}\to G$ in terms of the commutative diagram
\begin{equation}
\label{eq:pg01}
\xymatrix{
G_\II \ar[r]^{g\mapsto g^{-1}\cdot 0}\ar[d]& \A_{\II}\ar[d]^-{\Hol}\\
G\times G \ar[r]_{}& G
}
\end{equation}
where the left vertical map is given by $g\mapsto (g(0),g(1))$ and the lower horizontal map is 
$(a_0,a_1)\mapsto a_0^{-1}a_1$.
Both horizontal maps may be seen as quotient maps for a principal $G$-action, given by multiplication from the left. All maps in the diagram are $G_\II$-equivariant, where $G_\II$ acts on itself by 
\[ g\mapsto k.g,\ \ (k.g)(t)=g(t)k(t)^{-1},\]
on $G\times G$ by $(a_0,a_1)\mapsto (a_0k(0)^{-1},\ a_1k(1)^{-1})$, 
and on $G$ by $a\mapsto k(0) a k(1)^{-1}$. In particular, 
\begin{equation}\label{eq:equivhol}
\Hol(k\cdot A) = k(0)\, \Hol(A)\, k(1)^{-1},
\end{equation}
for $k\in G_\II$. 
The map $\Hol$ may be regarded as the quotient map for the principal action of the subgroup 
\begin{equation}\label{strholfib}
G_{\II,\partial \II} = \{g\in G_\II\colon \ g(0) = g(1) = e\}.
\end{equation}
By taking the differential of \eqref{eq:equivhol}, we see that 
the differential $T\Hol\colon T\A_\II\to T G$ satisfies
\begin{equation}\label{eq:holdiffact}
(T_A\Hol)(\xi_{\A_{\II}}|_A) 
= (\xi(1)^L - \xi(0)^R)|_{\Hol(A)} 
\end{equation}
for  $\xi \in \g_\II$.

\subsection{Principal connections for the holonomy fibration
 }\label{subsec:princon}
Any function $\chi\in C^\infty(\II)$ with $\chi(0)=0$ and $\chi(1)=1$ defines a connection $\theta$ on the principal bundle $\A_\II\to G$.  The connection can be described in terms of the 
corresponding horizontal lift. Let $g\in G_\II$ be any path such that $A=g^{-1}\cdot 0$. 
Using left-trivialization $TG=G\times \g$, the horizontal lift for $\theta$ is given as 
\[ T_{\Hol(A)}G\to T_A\A_\II,\ \ X\mapsto \partial_A \xi\]
where $\xi\in\g_\II$ is the path 
\[ \xi(t)=\chi(t)\Ad_{g(t)^{-1} g(1)}X.\] 
Note that this does not depend on the choice of $g$ with $g^{-1}\cdot 0=A$.
In Appendix \ref{app:cal}, we review the `conceptual construction' of $\theta$, provided by the \emph{caloron correspondence}.  The horizontal bundle defined by $\theta$ is invariant under the full action of $G_\II$ (not only of the structure group $G_{\II,\p\II}$ of the principal bundle). 

In particular, one can take $\chi(t)=t$. The resulting connection $\theta$ is uniquely  characterized by $G_\II$-invariance together with the value at the zero connection 
$A=0$,  given by
\begin{equation}\label{eq:thetaA}
\iota(a)\theta(t)=\int_0^ta - t \int_0^1 a, \ \ a \in T_0\A_\II.
\end{equation}
That is, the horizontal space at $A=0$ is $\g\subseteq \Omega^1(\II,\g)$, embedded as `constant 1-forms'.

\begin{remark}
Suppose $r=0$, so that $\A_\II$ consists of $L^2$-connections, and suppose that $\g$ comes equipped with an $\Ad$-invariant metric (as in Section \ref{sec: Courant A} below).  Define a $G_\II$- invariant pseudo-Riemannian metric on $\A_\II$ via
  \begin{equation}\label{prmai}
  (a_1,a_2)=\int_{[0,1]} a_1\cdot *a_2.
  \end{equation}
 Then the connection $\theta$ defined by $\chi(t)=t$ is the unique connection for which 
 the horizontal spaces are orthogonal to the $G_{\II,\partial\II}$-orbits for the metric~\eqref{prmai}. 
(This is easily verified at $A=0$; the claim follows by invariance.) 
\end{remark}

\section{Dirac reduction for the holonomy fibration}
\label{sec: Courant A}
Let $G$ be a connected finite-dimensional Lie group with a bi-invariant pseudo-Riemannian metric, so that its Lie algebra $\g$  is a \emph{metrized Lie algebra}, that is, it comes with 
 a non-degenerate $\Ad$-invariant symmetric bilinear form, denoted by $(X_0,X_{1})\mapsto X_{0}\cdot X_{1}$.  We use the metric to define a $G_\II$-invariant ``Lie-Poisson'' structure on $\A_\II$, a weak Poisson structure given by an invariant Dirac structure in the standard Courant algebroid $\T\A_\II$. 
We then explain how to carry out a reduction along the holonomy map $\Hol\colon \A_\II\to G$, obtaining the Cartan-Dirac structure of Section \ref{sec:carcou}.
We also consider more general weak Poisson structures on $\A_\II$, which reduce to other 
Dirac structures on $G$.  Finally, we study the reduction of Hamiltonian spaces for these weak Poisson structures.

\subsection{$G_\II$-action on $\T\A_\II$} 
\label{subsec:PGaction}
Let $\d A$ denote the tautological $\Omega^1_{\HH_r}(\II,\g)$-valued 1-form on the affine space 
$\A_\II$, defined by 
\[ \iota(a)\, \d A=a\] 
for all $a\in T_A\A_\II$. Let $\T \A_\II=T\A_\II\oplus T^\star\A_\II$ be the standard Courant algebroid
over $\A_\II$, and define sections 
%
\begin{equation}\label{eq:varrho}
\varrho(\xi)=\xi_{\A_{\II}}+ \l\d A,\, \xi\r \in \Gamma(\T \A_\II),\ \ \
\xi\in \g_\II ,
\end{equation} 
where $\xi_{\A_{\II}}$ are the generating vector fields for the $G_\II$-action
on $\A_\II$, and the 1-form component is such that 
\[
i_{a}\l\d A,\,\xi\r = \int_{I}a\cdot \xi,\ \ \ a\in T_{A}\A_{\II}.
\]
Note that this is similar to the formula for the sections spanning the Lie-Poisson structure on $\g^\star$ (cf.~ Equation \eqref{eq:coadj}). 
For any subspace ${\mf{s}}\subseteq \ol{\g}\oplus\g$, let
\[ \g_\II^{(\s)}=\{\xi\in \g_\II|\ (\xi(0),\xi(1))\in {\mf{s}}\}\]
be the subspace of paths with end points in $\mf{s}$. Let
$\E^{(\mf{s})}\subseteq \T\A_\II$ denote the subbundle spanned by all
$\varrho(\xi),\ \xi\in \g_\II^{(\s)}$. For the trivial subspace $\mf{s}=\{0\}$, the space $\g_\II^{(\s)}$ coincides with $\g_{\II,\partial\II}$, the Lie algebra of the structure group~\eqref{strholfib} of the holonomy fibration.
\begin{proposition}\label{prop:gene}
The sections \eqref{eq:varrho} are generators for the standard lift of the $G_\II$-action 
to the Courant algebroid $\T\A_\II$.  They satisfy 
  \begin{equation}\label{eq:innpr}
    \l\varrho(\xi),\varrho(\zeta)\r=\xi(1)\cdot\zeta(1)-\xi(0)\cdot\zeta(0)
  \end{equation}
  for all $\xi,\zeta\in \g_\II$. Furthermore, for any subspace $\s\subset  \ol{\g}\oplus\g$, 
  one has 
   \begin{equation}\label{eq:perp}
(\E^{(\s)})^\perp=\E^{({\mf{s}}^\perp)}.
  \end{equation}
\end{proposition}

\begin{proof}
The map $\g_\II\to \Omega^1(\A_\II),\ \xi\mapsto \l\d A,\,\xi\r$ is $G_\II$-equivariant and 
takes values in closed 1-forms. 
Since the $\xi_{\A_{\II}}$ (viewed as sections of $\T\A_\II$) are generators for the action, 
so are $\xi_{\A_{\II}}+ \l\d A,\,\xi\r$. In particular, 
\begin{equation}\label{eq:equivariance}
\varrho(\Ad_g \xi)=g.\varrho(\xi)\end{equation}
for all 
$g\in G_\II,\ \xi\in\g_\II$.
Furthermore, 
\[ \begin{split}
\l \varrho(\xi),\varrho(\xi)\r&=
2\iota(\xi_{\A_{\II}})\l\d A,\,\xi\r
=2 \int_0^1 \partial_A\xi\cdot \xi =\int_0^1 \partial (\xi\cdot \xi)
\\&=\xi(1)\cdot\xi(1)-\xi(0)\cdot\xi(0),
\end{split} \]
which proves \eqref{eq:innpr} by polarization. 

This also gives the reverse inclusion in \eqref{eq:perp}. For the forward inclusion, 
we first show that
 \begin{equation}\label{eq:perp1}
  \varrho(\A_\II\times \g_{\II,\partial\II})^\perp\subset \varrho(\A_\II\times \g_\II).\end{equation}
We may use the $G_\II$-invariance to assume $A=0$. Suppose $b+\beta\in \T_0\A_\II$ is orthogonal to $\varrho(A_\II\times \g_{\II,\partial\II})$. That is, 
for all $\xi \in \g_{\II,\partial\II}$, 
\[ 0=\l b+\beta,\ \xi_{\A_{\II}}+  \l\d A,\,\xi\r\r=
\beta(\partial \xi)+\int_0^1 b\cdot\xi.\]
Let $\zeta$ be a solution of $\partial\zeta=b$, with the unique initial condition $\zeta(0)$ such that for all $X\in\g$,
\[ X\cdot\int_0^1 \zeta(t)\,dt=\beta(X\,d t).\]
By elliptic regularity, $\zeta$ has Sobolev class $r+1$, so that $\zeta\in \g_\II$. We will show
that $\beta=\l\d A,\zeta\r$, which then proves that $b+\beta=\varrho(\zeta)|_0$. 

Consider the decomposition of $T_0\A_\II$ into horizontal and vertical directions, relative to the standard connection $\theta$ given by \eqref{eq:thetaA}. The vertical space is spanned by elements $\xi_{\A_{\II}}=\partial\xi$
with $\xi\in \g_{\II,\partial\II}$, and we have 
\[ \iota(\xi_{\A_{\II}})\l \d A,\zeta\r=\l \partial\xi,\zeta\r=-\l\partial\zeta,\xi\r=-\int_0^1 b\cdot\xi=\beta(\partial\xi).\]
The horizontal space is spanned by elements of the form $X\, d t$ with $X\in \g$, and on such elements we have 
\[ \iota(Xd t)\l \d A,\zeta\r=X\cdot \int_0^1 \zeta(t)\d t=\beta(X d t),\]
by definition of $\zeta$, establishing~\eqref{eq:perp1}. 
For the final equality \eqref{eq:perp}, note that  $\varrho(\A_\II\times \g_{\II,\partial\II})\subseteq {\E^{(\s)}}$, and so, using~\eqref{eq:perp1}, 
$({\E^{(\s)}})^\perp\subseteq \varrho(\A_\II\times \g_\II)$.  Then~\eqref{eq:perp} follows from  \eqref{eq:innpr}.
\end{proof}

\subsection{The Lie-Poisson structure on $\A_\II$}
\label{subsec: weak Poisson}
Proposition \ref{prop:gene} shows that  ${\E^{(\s)}}$ is a Lagrangian subbundle if and only if the subspace $\mf{s}$ is Lagrangian. Since the map $\varrho\colon \g_\II\to \Gamma(\T \A_\II)$ is bracket preserving, it follows that ${\E^{(\s)}}$ is a Dirac structure if and only if
$\mf{s}$ is a Lagrangian Lie subalgebra. 

\begin{remark}
Since  $\xi_{\A_{\II}}|_A=\partial_A \xi$,  
the fiber ${\E^{(\s)}}|_A$ may be interpreted as the  graph of the  unbounded operator $\partial_A$ with dense domain 
$\g_I^{(\s)}\subseteq \Omega^0_{\HH_{r+1}}(\II,\g)$
consisting of paths of Sobolev class $r+1$ with end points in $\s$. Taking $r=0$, the operator $\partial_A$ is skew-adjoint if and only if $\s$ is Lagrangian. 
\end{remark}

\begin{proposition}\label{lem:C}
  For any Lagrangian Lie
  subalgebra $\s$, the Dirac structure ${\E^{(\s)}}$ is a weak Poisson structure on
  $\A_\II$.  Its leaves $\O$ are the orbits of the
  $\g_\II^{(\s)}$-action on $\A_\II$, with weakly symplectic 2-forms
  $\om_\O$ given on generating vector fields by
  \[ \om_\O(\xi_{1,\O},\,\xi_{2,\O})|_A= \int_{I}
  \xi_1\cdot \partial_A \xi_2,\] for $A\in \O$ and $\xi_1,\xi_2\in \g_\II^{(\s)}$.
\end{proposition}
\begin{proof}
  From the formula $\varrho(\xi)=\xi_{\A_{\II}}+\l\d A,\,\xi\r$, it is immediate
  that ${\E^{(\s)}}\cap T\A_\II=0$, and that the orbits of
  ${\E^{(\s)}}$ are the orbits of the $\g_\II^{(\s)}$-action. By
  definition, the 2-forms on these orbits satisfy
  \[ \iota(\xi_\O)\om_\O|_A=-i_\O^*\l\d A,\,\xi\r    ,\ \ \ \xi\in
  \g_\II^{(\s)},\ A\in\O\]
  where $i_\O$ is the inclusion of $\O$. Hence
  \[ \om_\O(\xi_{1,\O},\,\xi_{2,\O})|_A=-\iota(\xi_{2,\O})\l \d
  A,\xi_1\r =\l \partial_A \xi_2,\xi_1\r =\int_I \xi_1\cdot \partial_A
  \xi_2,\ \ \]
  for all $\xi_1,\xi_2\in \g_\II^{(\s)}$.
\end{proof}
The case of the  diagonal 
$\mf{s}=\g_\Delta$ (periodic boundary conditions)
is particularly important.  The Dirac structure $\E\equiv\E^{(\g_\Delta)}$ is called the \emph{Lie-Poisson
  structure} on  $\A_\II$.
\begin{remark}  
By definition, the algebra of admissible functions (cf.~ Section \ref{subsec:weakpoisson}) for the weak Poisson structure $\E^{(\s)}$ contains all affine-linear functions of the form $f(A)=t+\l A,\xi\r$ with 
$\xi\in \g_\II^{(\s)}$ and $t\in \R$;  the corresponding Hamiltonian vector field is $v_f=\xi_{\A_{\II}}$. 
These affine-linear functions form a Lie algebra under the Poisson bracket:
\[ \{t_1+\l A,\xi_1\r,\ t_2+\l A,\xi_2\r\}= \L(\xi_{1,\A})\l A,\xi_2\r=\int_\II \partial_A \xi_1\cdot \xi_2
=\int_\II \partial \xi_1\cdot\xi_2+\l A, [\xi_1,\xi_2]\r.\]
Therefore, they define a central extension of the Lie algebra $\g_\II^{(\s)}$, with cocycle
$\int_\II \partial\xi_1\cdot\xi_2$. For 
$\s=\g_\Delta$, this is the standard central extension of  the loop algebra. 
\end{remark}

Let $S\subseteq D= \ol{G}\times G$ be a Lie subgroup whose Lie algebra $\mathfrak{s}\subseteq \dd=\ol{\g}\oplus \g$ is Lagrangian. Consider the subgroup 
\[ \ca{S}\equiv G_\II^{(S)}\subseteq G_\II\] 
consisting of paths $g\in G_\II$ with endpoints $(g(0),g(1))\in S$.   Generalizing \cite[Theorem 8.3]{al:mom}, we have: 
\begin{proposition}\label{prop:HamA}
An exact Hamiltonian $\S$-space for the weak Poisson structure $(\T\A_\II,\E^{(\mathfrak{s})})$ is equivalent to a manifold $\M$ with an action of $\S$, together with an invariant weakly symplectic 2-form
$\sigma\in \Omega^2(\M)$ and an equivariant \emph{moment map} $\Psi\colon
\M\to\A_\II$ satisfying
\begin{equation}\label{eq:momcond} \iota(\xi_\M)\sigma=-\l \d
  \Psi,\xi\r,\ \ \ \xi\in \g^{(\s)}_\II.\end{equation}
Here $\l\d\Psi,\xi\r\in \Omega^1(\M)$ denotes the pullback by $\Psi$ of the 1-form
$\l\d A,\,\xi\r\in\Omega^1(\A_\II)$. 
\end{proposition}
\begin{proof} 
This is a special case of Proposition \ref{prop:hamiltonian} together with Example \ref{ex:exex}. 
\end{proof}

\subsection{Reduction of the Lie-Poisson structure on the space of connections}
\label{subsec: A reduction}
In this section, we exhibit the Cartan-Dirac structure from Section \ref{sec:carcou}
as a reduction of the Lie-Poisson structure on the space  of connections over the unit interval. 
In Section \ref{sec:conns1}, we give a similar construction for connections over the circle $S^1$.  Since the standard lift of the 
principal $G_{\II,\partial\II}$-action on $\A_\II$ to $\T\A_\II$ has isotropic generators, we use the machinery of Section \ref{subsec:coured} to define a reduced Courant algebroid  $(\T\A_\II)_{red}$ over $G=\A_\II/G_{\II,\partial\II}$. 

\begin{theorem} \label{thm:redu} 
  The reduced Courant algebroid $(\T\A_\II)_\red$ is canonically isomorphic to the Cartan-Courant algebroid $\AA$ over $G$. This isomorphism intertwines the $G\times G\cong G_\II/G_{\II,\partial\II}$-actions together with their generators, and 
restricts to an isomorphism of Dirac structures
\[ ((\T\A_\II)_\red,({\E^{(\s)}})_\red)\cong  (\AA,E^{(\s)})\]
for each Lagrangian Lie subalgebra $\mf{s}\subseteq \dd$. 
In particular, the reduction of the Lie-Poisson structure on $\A_\II$ is the Cartan-Dirac structure on $G$. 
Also, the $G_\II$-basic splitting of $\T \A_\II$ defined by a  principal connection $\theta$ as in Section~\ref{subsec:princon} 
descends to the splitting \eqref{eq:courantsplitting} of the Cartan-Courant algebroid. 
\end{theorem}
%
\begin{proof}
The map $G_\II\to G\times G,\ k\mapsto (k(0),k(1))$  descends to an identification
\[ G_\II/G_{\II,\partial\II}=G\times G,\ \ \  \g_\II/\g_{\II,\partial\II}=\ol\g\oplus \g.\]
Let $C=\varrho(\A_\II\times \g_\II)$, thus  $C^\perp=\varrho(\A_\II\times \g_{\II,\partial\II})$
by  \eqref{eq:perp}. By definition,  $(\T\A_\II)_\red=(C/C^\perp)/G_{\II,\partial\II}$. Since the action of $G_{\II,\partial\II}$ on 
$\g_\II/\g_{\II,\partial\II}$ is trivial, it follows that $(\T\A_\II)_\red$ is an action Courant algebroid 
\[ (\T\A_\II)_\red=G\times \g_\II/\g_{\II,\partial\II}=G\times (\ol{\g}\oplus \g)\]
with the constant sections as the reduced generators  $\varrho_\red\colon  \g_\II/\g_{\II,\partial\II}\to \Gamma(\T\A_\II)_\red$. The action of $[k]\in G_\II/G_{\II,\partial\II}$ on $G$ is induced from the action of $k\in G_\II$ on $\A_\II$, and is given by $[k].a=k(0)ak(1)^{-1}$, by the equivariance property \eqref{eq:equivhol} of the holonomy map. This shows that 
$ (\T\A_\II)_\red$ is the Cartan-Courant algebroid $\AA$, where the isomorphism intertwines the actions and the generators. 
Since $\g_\II^{(\s)}/\g_{\II,\partial\II}=\s$, it is immediate that  ${\E^{(\s)}}=\varrho(\A_\II\times \g_\II^{(\s)})$ has reduction $E^{(\s)}=G\times \mf{s}$.

We now verify the reduction of splittings.
As in Section \ref{subsec:exact-red}, let $\varpi\in \Omega^2(\A_\II)$ be the 2-form determined by the principal connection $\theta$. 
In the notation from that section,
\begin{equation}\label{eq:alphac}
 \alpha(\xi)=\l\d A,\,\xi\r,\ \ \ c(\xi,\xi')=\iota_{\xi_{\A_{\II}}}\alpha(\xi')=\l\partial_A\xi,\xi'\r,
\end{equation}
for $\xi,\xi'\in \g_{\II,\partial\II}$, hence 
\[ \varpi=- \l\d A,\,\theta\r+\hh \l \partial_A\theta,\, \theta\r,\]
defining the $G_{\II,\partial\II}$-basic splitting 
$ j\colon T\A_\II\to \T \A_\II$ via $j(a)=a+\iota(a)\varpi$. 
Let $j_\red\colon TG\to (\T\A_\II)_\red$ the reduced splitting. 
To compute it, let $\beta\colon \g \to \Omega^1(G)$ be the map given as 
\begin{equation}\label{eq:A1}
\varrho_\red(0,X)-j_\red(X^L)=\a_\red^*(\beta(X))\end{equation}
for all $X\in\g$, with $\a_\red\colon  (\T\A_\II)_\red=\AA\to TG$ the reduced anchor. 
Then 
\begin{equation}\label{eq:A2} \varrho(\xi)-j(\xi_{\A_{\II}})=\Hol^*\beta(\xi(1))\end{equation}
for all $\xi\in\g_\II$ with $\xi(0)=0$. We use \eqref{eq:A2} to compute the map $\beta$, which then determines 
$j_\red$ via \eqref{eq:A1}.  
Let $\theta$ be obtained from the function $\chi\in C^\infty(\II)$ with $\chi(0)=0,\ \chi(1)=1$, as in Section \ref{subsec:princon}. 
Given $X,\,Z\in\g$, let 
\[ \xi(t)=\chi(t)\Ad_{g(t)^{-1}g(0)}X,\ \ \ \ \ \zeta(t)=\chi(t)\Ad_{g(t)^{-1}g(0)}Z.\]
Then $\xi,\zeta$ are the unique paths from $0$ to $X,Z$ such that  $\xi_{\A_{\II}}|_A,\ \zeta_\A|_A$ are  horizontal with respect to $\theta|_A$.  
With this choice of $\xi$, we obtain 
\[ \iota(\xi_{\A_{\II}})\varpi=-\l \xi_{\A_{\II}},\theta\r=-\l \partial_A\xi,\theta\r,\]
hence $j(\xi)=\xi_{\A_{\II}}-\l \partial_A\xi,\theta\r$. It follows that 
$\Hol^*\beta(X)= \l\d A,\,\xi\r     -\l \partial_A\xi,\theta\r$,
thus
\begin{equation}
\iota(\zeta_\A)\Hol^*\beta(X)=
 \l \partial_A \zeta,\xi\r
=Z\cdot X \int_0^1 \f{\p \chi}{\p t}\chi(t)dt=\hh Z\cdot X.
 \end{equation}
Since $\zeta_\A\sim_{\Hol} Z^L$, the left hand side can also be written 
$\Hol^*\iota(Z^L)\beta(X)$. We conclude $\beta(X)=\hh X\cdot \theta^L$, and hence 
\( j_\red(X^L)=\varrho_\red(0,X)-\hh \a_\red^* \theta^L\cdot X.\)
This is consistent with the formulas \eqref{eq:courantsplitting} for the Cartan-Courant algebroid, proving that the two splittings coincide. 
\end{proof}

\begin{samepage}
\begin{remark}\label{rmks: red}\mbox{}
  \begin{enumerate}
  \item The above theorem holds for all regularities $r\geq 0$ imposed
    on the connections $\A_\II$. It thus shows that the reduction
    $(\T\A_\II)_\red$ is insensitive to the chosen regularity $r\geq 0$.

\item 
 As shown in \cite{al:ati}, the 2-form $\varpi\in \Omega^2(\A_\II)^{G_\II}$ determined by the standard connection $\theta$ 
 on the holonomy fibration is given by the formula
 \[ \varpi=\hh \int_{[0,1]} \Hol_s^*\theta^R \cdot \f{\p}{\p
    s}(\Hol_s^*\theta^R) \d s \in \Omega^2(\A_\II)^{G_\II},\]
    where $\theta^R \in \Omega^1(G,\g)$ is the right invariant Maurer-Cartan $1$-form on $G$, 
and $\Hol_s\colon \A_\II\to G$ is given by $\Hol_s(A)=g(s)$, where $g\in G_{\II}$ is the parallel transport  for $A$, i.e.~ $g(0)=e$ and $A=g^*\theta^L$.  This $2$-form $\varpi$ also appears in \cite[Section 8.1]{al:mom}. 
\item The $(G\times G)$-equivariant splittings of the Cartan-Courant algebroid form an affine space for the vector space of bi-invariant 2-forms on the base $G$. If $G$ is compact 
or semi-simple, then the space $\Omega^2(G)^{G\times G}=(\wedge^2\g^*)^G$ is zero. Hence, in this case 
\emph{any} $G_\II$-invariant connection 1-form $\theta$ on $\A_\II$ will lead to the same 
2-form $\varpi$, and to the same reduced 
splitting of 
$(\T \A_\II)_\red=\AA$. 
\end{enumerate}
\end{remark}
\end{samepage}



\subsection{Reduction of Hamiltonian spaces} 
\label{subsec: ham LG} 
Let $S\subseteq D= \ol{G}\times G$ be a Lie subgroup whose Lie algebra $\mathsf{s}\subseteq \dd=\ol{\g}\oplus \g$ is Lagrangian. Consider the subgroup 
\[ \ca{S}\equiv G_\II^{(S)}\subseteq G_\II\] 
consisting of paths $g\in G_\II$ with endpoints $(g(0),g(1))\in S$. 
The group $\S$ contains $G_{\II,\partial\II}$ as a normal subgroup, with quotient 
$\S/G_{\II,\partial\II}=S$.
As a special case of the general result concerning reduction of  Dirac structures (Proposition \ref{prop:Diracred}), we obtain: 
\begin{proposition}\label{prop:LG}
Reduction by the action of $G_{\II,\partial\II}$ defines a 1-1 correspondence between 
 Hamiltonian $\S$-spaces $\M$ for $(\T\A_\II,{\E^{(\s)}})$ and 
 Hamiltonian $S$-spaces $M$ for $(\AA,E^{(\s)})$. The spaces and moment maps 
are related by the commutative diagram
\[ \xymatrix@C=8ex{ \M \ar[r]^\Psi\ar[d]^{\pi} & \A_\II \ar[d]^{\Hol} \\ M\ar[r]^{\Phi} & G }\]   
Here $\pi\colon \M\to M$ is the quotient by the action of $G_{\II,\partial\II}$. The correspondence preserves exactness. 
\end{proposition}
\begin{proof} 
For any $\S$-equivariant map $\Psi\colon \M\to \A_\II$, since the action of 
$G_{\II,\partial\II}\subset \S$ on $\A_\II$ is a principal action, the action on $\M$ is a principal action. Taking quotients by $G_{\II,\partial\II}$, one obtains a manifold $M$ with an  
$S=\S/G_{\II,\partial\II}$-equivariant map $\Phi$ to $G=\A_\II/G_{\II,\partial\II}$.
Conversely, given $M$ with an $S$-equivariant map $\Phi\colon M\to G$,  define $\M\subseteq M\times \A_\II$ as the pullback of the principal bundle $\Hol\colon \A_\II\to G$ under the map $\Phi$. The diagonal $\S$-action on $M\times \A_\II$ (where the action on $M$ is via the quotient map to $S$) restricts to an action on $\M$, and the projection to the second factor restricts to an $\S$-equivariant map $\Psi\colon \M\to \A_\II$. 
Suppose now that  
\[ \ca{R}\colon (\T\M,T\M)\da (\T\A,{\E^{(\s)}})\] 
is a Dirac morphism, 
with base map $\Psi$.  According to Proposition \ref{prop:Diracred}, the reduction by  $G_{\II,\partial\II}$ gives an $S$-equivariant Dirac morphism 
\[ R=\ca{R}_\red\colon 
(\T M,TM)\da ((\T\A_\II)_\red,({\E^{(\s)}})_\red)\cong  (\AA,E^{(\s)})\]
with base map $\Phi=\Psi_\red\colon M\to G$. By Proposition \ref{prop:Diracred1},  the morphism $\ca{R}$ is exact if and only if $R$ is exact. Conversely, given the  Dirac morphism $R\colon (\T M,TM)\da (\AA,E^{(\s)})$, part (c) of Theorem \ref{th:Rred} shows how to recover $\ca{R}$. 
\end{proof}
Note that if the moment map $\Psi$ is proper, then so is $\Phi$. In this case, 
the finite-dimensionality of $G$ implies finite-dimensionality of $M$. 
 
Recall that the exact Hamiltonian spaces for $(\T\A_\II,\E^{(\s)})$ are described by triples $(\M,\sigma,\Psi)$ (see Proposition \ref{prop:HamA}), while those for $(\AA,E^{(s)})$ are described 
by triples $(M,\om,\Phi)$ (see Proposition \ref{prop:qh}). Under the correspondence from Proposition \ref{prop:LG}, these are related as follows. 
Let $\varpi\in \Omega^2(\A_\II)$ be the $G_\II$-invariant $2$-form defined by the standard connection $\theta$ on  the holonomy fibration. 
\begin{proposition}
Let $(\M,\sigma,\Psi)$ be an exact Hamiltonian $\S$-space for $(\T\A_\II,\E^{(\s)})$,  and $(M,\om,\Phi)$ the corresponding exact Hamiltonian $S$-space for $(\AA,E^{(s)})$. Then  
\begin{equation}\label{eq:omsigma} \sigma=\pi^*\om+\Psi^*\varpi.\end{equation}
\end{proposition}
\begin{proof}
In terms of the splittings, we have $\ca{R}=\T\Psi_\sigma$ and $R=\T\Phi_\om$, for an 
$\S$-invariant 2-form $\sigma\in \Omega^2(\M)$ and an $S$-invariant 2-form $\om\in \Omega^2(M)$. 
Since the $\varpi$-twist of the standard splitting of $\T\A_\II$
descends to the splitting \eqref{eq:courantsplitting} of $\AA$, these 2-forms are related by 
\eqref{eq:omsigma}. 
\end{proof}

\subsection{Multiplicative structures}\label{subsec:connmult}
\alejandro{This is the new subsection obtaining $\Mult_\AA$ from connections on the Disk.}
In this subsection, we obtain the multiplicative structures $\Mult_\AA$ and $\Inv_\AA$ on the Cartan-Courant algebroid $\AA$ described in Section   \ref{subsec:mult} as a reduction from appropriate spaces of connections.

We begin describing how to get group multiplication $\Mult_G: G\times G \to G$ in terms of spaces of connections. Let $\M$ denote the space of \emph{flat} $G$-connections of class $\HH_k$ on the trivial principal $G$-bundle over a triangle $\DD \subset \R^2$ (i.e. a $2$-simplex), with $k>1$. Following \cite[Section 9.1]{al:mom}, $\M$ is a smooth infinite dimensional Hilbert manifold on which the Hilbert Lie group $G_\DD = Map_{\HH_{k+1}}(\DD,G)$ acts by gauge transformations.

Let $z_0,z_1,z_2\in \bd \DD$ be the cyclically oriented  vertices of the 2-simplex. ($\bd \DD$ is taken positively oriented w.r.t. $\DD$.) We thus define a map 
\[ \Phi: \M \to \A_\II \times \A_\II \times \A_\II, A \mapsto ( \bar{\gamma}_2^*A, \gamma_0^*A, \gamma_1^*A)\]
where $\gamma_i:[0,1] \to \bd \DD$ is an orientation preserving parameterization of the edge $[z_i,z_{i+1}]\subset \bd \DD$, for $i=0,1,2$ ($z_3=z_0$), and we denoted $\bar{\gamma}(t)=\gamma(1-t)$. Here, we take $\A_\II$ with regularity $r=k-1/2$ so that $\Phi$ is smooth because $k>1$. If we consider the subgroup $G_{\DD,Z} = \{ g \in G_\DD : g(z_i)=e\}$ acting on $\M$ and $G_{\II,\bd \II} \times G_{\II,\bd \II}\times G_{\II,\bd \II}$ acting on $(\A_\II)^3$, the map $\Phi$ is equivariant relative  to the group homomorphism $f:g \mapsto ( \bar{\gamma}_2^*g, \gamma_0^*g, \gamma_1^*g)$. The induced map 
\[\Phi_\red:M:=\M/G_{\DD,Z} \to (\A_\II)^3/(G_{\II,\bd \II})^3 \simeq G \times G \times G\] 
is an embedding of $M\simeq G^2$ inside $G^3$ satisfying 
\begin{equation}\label{eq:multG}
 \Phi_\red(M) = \on{gr}(\Mult_G)=\{(k,g,h)\in G^3: ghk^{-1}=e\},
\end{equation}
since the holonomy around $\bd \DD$ of a flat connection on $\DD$ is trivial.


At the level of Courant algebroids, the map $\Phi$ can be supplemented with the Atiyah-Bott presymplectic $2$-form $\sigma \in \Omega^2(\M)$ (\cite{at:mo}). It has the following property (see e.g. \cite[Section 9.1]{al:mom}), for $\xi \in \g_\DD = \Omega^0_{\HH_{k+1}}(\DD,\g)$ inducing the infinitesimal gauge transformation $\xi_\M|_A \in T_A \M$,
\begin{equation}\label{eq:sig}
 i_{\xi_\M|_A} \sigma = -\int_{\bd \DD} A \cdot \xi = \int_I  \bar{\gamma}_2^*(A\cdot \xi) - \int_I  \gamma_0^*(A\cdot \xi)-\int_I  \gamma_1^*(A\cdot \xi).
\end{equation}
The induced exact Courant morphism
\[ \T\Phi_\sigma : \T \M \da \T \A_\II \times  \T \A_\II \times  \T \A_\II\]
is thus equivariant relative to $f$ when considering the natural lifted $G_{\DD,Z}$-action on $\T \M$ and the $(G_{\II,\bd \II})^3$-action defined by $\varrho \times \bar{\varrho}\times \bar{\varrho}$ on $(\T \A_\II)^3$. (Here $\bar{\varrho}(\eta) = \eta_\A - \langle \d A,\eta\rangle$, for $\eta \in \g_{I}$, corresponds to the generators associated to opposite metric on $\g$.)
We can thus apply Thm \ref{th:Rred} and reduce the (exact) Courant morphism $\T\Phi_\sigma$ to a (exact) Courant morphism 
\[ (\T\Phi_\sigma)_\red: \T M \da \AA \times \bar{\AA} \times \bar{\AA}.\] The Courant analogue of eq. \eqref{eq:multG} is the following:

\begin{proposition}
 With the notations above,
 \[\on{gr}(\T\Phi_\sigma)_\red \circ TM = \on{gr}(\Mult_\AA),\]
 where $\Mult_\AA:\AA \times \AA \da \AA$ was defined in Section \ref{subsec:mult}.
\end{proposition}

\begin{proof}
We shall denote $q_M : \T \M \da \T M$ and $q_\AA : \T \A_\II \da \AA$ the quotient relations and $R=(\T\Phi_\sigma)_\red$. (Recall that $\varrho(\xi_\A|_A) \sim_{q_\AA} (\on{Hol}(A), \xi(0)\oplus \xi(1))$.) It is clear that $R \circ TM = R \circ q_M \circ T\M = (q_{\AA}\times q_{\bar{\AA}} \times q_{\bar{\AA}})\circ \T\Phi_\sigma \circ T\M$ at the set-theoretic level.
 Since the r.h.s. in the Proposition is Lagrangian, we only need to show that $R\circ TM$ is included in this set. This, in turn, follows from the fact that, given $g,h \in G$ and $X_i \in \g, i=0,1,2,$ one can find $A\in \M$ and $\xi \in \g_\DD$ so that $[A]\simeq (gh,g,h)$ and $\xi(z_i)= X_i$. For, then, using eq. \eqref{eq:sig}, $\xi_\M|_A \in T\M$ is related by $(q_{\AA}\times q_{\bar{\AA}} \times q_{\bar{\AA}})\circ\T\Phi_\sigma $ to $(gh, X_0\oplus X_2) \times (g,X_0 \oplus X_1) \times (h, X_1 \oplus X_2)$ as wanted. 
\end{proof}

\begin{remark}
The basic splitting of $\T\A_\II$ given in Thm. \ref{thm:redu} can be used to induce a splitting of $\T(\A_\II)^3$ which is basic for $\varrho\times\bar{\varrho}\times\bar{\varrho}$.
Following Prop. \ref{prop:reductionofexactmorphisms}, the reduced splitting takes the reduced exact Courant morphism $(\T\Phi_\sigma)_\red$ to the form $\T \Phi_{\red,\sigma_{red}}$ for an induced $2$-form $\sigma_\red \in \Omega^2(\M_\red)$. Using the identification $\M_\red\simeq G^2, \ [A]\mapsto (\Hol(\gamma_0^*A),\Hol(\gamma_1^*A))$, a straighforward computation shows that $\sigma_\red = \varsigma\in \Omega^2(G\times G)$, the $2$-form introduced in eq. \eqref{eq:varsigma}.
\end{remark}

Finally, we describe the inversion morphism $\Inv_\AA$ as a reduction. The diffeomorphism $\Inv_\II:\A_\II \to \A_\II$, $a(s)ds \mapsto -a(1-s)ds$ is $G_{\II,\bd \II}$-equivariant  with respect to the group homomorphism $g(s)\mapsto g(1-s)$ and covers the group inversion $\Inv_G:G\to G$ along the holonomy fibration $\Hol:\A_\II \to G$. Moreover, the natural lift $\Inv_\A:\T\A_\II \to \T\A_\II$ of $\Inv_\II$ is equivariant for the  actions $\varrho$ and $\bar{\varrho}$, respectively. Recalling the definition of the quotient relation $q_\AA:\T\A_\II \da \AA$ as in the Proof above, the corresponding reduced morphism $(\Inv_\A)_\red:\AA \da \ol{\AA}$ relates
 \[ (g,X_0 \oplus X_1) \sim (g^{-1}, X_1 \oplus X_0).\]
Then $(\Inv_\A)_\red = \Inv_\AA$ coincides with inversion in the groupoid $G\times \dd$ as described in Section \ref{subsec:mult}.

\section{Connections over $S^1$}\label{sec:conns1}
\alejandro{I shortened a bit this first paragraph, in view of the second one}In the previous section, we obtained the Cartan-Courant algebroid on $G$, together with its Cartan-Dirac structure, by reduction along the principal $G_\II$-bundle $\Hol: \A_\II\to G$ for connections on a unit interval.  For applications to moduli spaces of flat connections over surfaces with boundary, one is interested in a modification of this construction using the space of connections over a circle, denoted by $\A_{S^1}$.  
In this case, the group acting on $\A_{S^1}$ is the loop group $LG$ and, unlike the $G_I$-action on $\A_\II$, this action is not transitive.

\alejandro{I filled this in, as a first approximation}In section~\ref{holcircle},  we describe an $L_0G$-bundle $\Hol: \A_{S^1} \to G$ corresponding to the quotient by the based loop group $L_0G$ and introduce a transitive Lie algebroid $R$ over $\A_{S^1}$. Here, $L_0G$ plays the role of $G_{\II,\bd \II}$ and $R$ that of the transitive $g_I$-action on $\A_\II$.
In section~\ref{redbaslop}, we introduce  an $LG$-action on the standard Courant algebroid $\T \A_{S^1}$ and a weak Poisson structure $\E$ analogous to the Lie-Poisson structure on $\A_\II$. Finally, we show that reduction of $(\T\A_{S^1},\E)$ under the $L_0G$-action also yields the Cartan-Dirac structure $(\AA,E)$.

\subsection{The holonomy fibration for the circle}\label{holcircle}
Let $\A_{S^1}=\Omega^1_{\HH_{r}}(S^1,\g)$ be the space of connections on the trivial $G$-bundle over the circle $S^1=\R/\Z$. Let
\begin{equation}\label{eq:lg}
 LG=G_{S^1}:=\Map_{\HH_{r+1}}(S^1,G)
\end{equation}
be the \emph{loop group}; the subgroup $L_0G$ of loops with $\gamma(0)=e$ is the \emph{based loop group}. We then define the \emph{path space} (see Appendix \ref{app:cal} for its relation to the caloron correspondence, which also makes it clear that it is a Hilbert manifold)
\begin{equation}\label{eq:pg} \ca{P}G=\{g\in \on{Map}_{\HH_{r+1}}(\R,G)|\  g(t+1) g(t)^{-1}= g(1)g(0)^{-1}\text{ for all } t\}. 
\end{equation}
The loop group $LG$ acts on $\ca{P}G$ by $(k\cdot g)(t)=g(t)k(t)^{-1}$. This action is a 
principal action, with quotient map $g\mapsto g(1)g(0)^{-1}$. The principal action commutes with the 
$G$-action on $\ca{P}G$ by pointwise multiplication from the left; this action makes $\P G$ into an $LG$-equivariant principal $G$-bundle over $\A_{S^1}$, with quotient map $g\mapsto A=g^{-1}\cdot 0$. 
The holonomy 
$\Hol\colon \A_{S^1}\to G$ of a connection may be defined in terms of the commutative diagram
\begin{equation}
\begin{aligned}
\label{eq:pg1}
\xymatrix{
\P G \ar[r]^{g\mapsto g^{-1}\cdot 0}\ar[d]_q& \A_{S^1}\ar[d]^-{\Hol}\\
G\times G \ar[r]_{}& G
}
\end{aligned}
\end{equation}
where the left vertical map is given by $q\colon g\mapsto (g(0),g(1))$ and the lower horizontal map is 
$(a_0,a_1)\mapsto a_0^{-1}a_1$. The holonomy map has the equivariance property $\Hol(k.A)=\Ad_{k(0)}\Hol(A)$ for $k\in LG$ and 
$A\in \A_{S^1}$. The generating vector fields for the action of $L\g=\Omega^0_{\HH_{r+1}}(S^1,\g)$ are again given by the covariant derivatives, 
\begin{equation}\label{eq:cova}
 \xi_{\A_{S^{1}}}|_A=\partial_A \xi;\end{equation} 
the differential of $\Hol$ maps these to the 
generators for the conjugation action. We denote by
\[ \pi\colon \P G\to G,\ \ g\mapsto g(0)^{-1}g(1)\]
 the map defined by the commutative diagram; it is the quotient map for the $G\times L_0G$-action 
 (not to be confused with the quotient map for the $LG$-action). 

\begin{lemma}
The tangent fiber to $\P G$ at $g$ has the following description
\begin{equation}\label{eq:tgdesc}
 T_g\ca{P}G\cong \Big\{\xi\in \Omega^0_{\HH_{r+1}}(\R,\g)|\  
 \Ad_{g(t)}(\xi(t+1)-\xi(t))=\on{const}\Big\}.
 \end{equation}
The action $T_g\ca{P}G\to T_{gh^{-1}}\ca{P} G$ of elements $h\in LG$ is given by 
$\xi\mapsto \Ad_h \xi$, while the action $ T_g\ca{P}G\to T_{ag}\ca{P} G$ of elements $a\in G$ 
is $\xi\mapsto \xi$.  In term of this identification \eqref{eq:tgdesc}, and using left trivialization  $TG=G\times \g$, 
the tangent map to the left vertical map in \eqref{eq:pg1} is given by 
\[ T_g q\colon T_g\P G\to   \g\oplus \g,\ \ \xi\mapsto q(\xi):=(\xi(0),\xi(1)).\]
\end{lemma}
\begin{proof}
The tangent bundle of $\P G$ can itself be regarded as the total space of the path fibration for the tangent group $TG$: 
\[ T(\P G)=\P(T G).\]
Using left trivialization to identify $TG=G\times \g$, the group structure reads as 
$(a_1,X_1)(a_2,X_2)=(a_1a_2,\Ad_{a_1^{-1}}X_1+X_2)$, and $(a,X)^{-1}=(a^{-1},-\Ad_a X)$. 
Hence, the condition for a path $t\mapsto (g(t),\xi(t))$ to define an element of $\P( TG)$ is that
\[ \big(g(t+1),\xi(t+1)\big)\big(g(t),\xi(t)\big)^{-1}
=\big(g(t+1)g(t)^{-1},\ \Ad_{g(t)}(\xi(t+1)-\xi(t)\big)\]
be constant as a function of $t$. The last claim follows since the tangent map to $q\colon \P G\to G\times G$ is the
corresponding map for $\P(TG)\to TG\times TG$ for the group $TG$. 
\end{proof}
Regard $\P G$ as an $LG$-equivariant principal $G$-bundle over $\A_{S^1}$, and let 
\[ R=T(\P G)/G\to \A_{S^1}\]
be the corresponding $LG$-equivariant Lie algebroid. 
\begin{proposition}
The fibers of the Lie algebroid $R$ have the following description, 
\begin{equation}\label{eq:otherat}
R_A= \big\{\xi\in \Omega^0_{\HH_{r+1}}(\R,\g)|\ \partial_A\xi \mbox{ is periodic }\big\},\end{equation}
with anchor map $\xi\mapsto \xi_{\A_{S^1}}(A)=\partial_A\xi$. 
The Lie bracket on sections of $R$ is given by 
\begin{equation}
[\xi_1,\xi_2]_R=[\xi_1,\xi_2]+\L({\xi_{1,\A}})\xi_2-\L(\xi_{2,\A}) \xi_1; 
\end{equation}
here $\L(a)\xi$ denotes the Lie derivative of the function $\xi$ with respect to the vector field $a$, and 
$ [\xi_1,\xi_2]$ is the pointwise Lie bracket. 
\end{proposition}
\begin{proof}
The subspace on the right hand side of \eqref{eq:tgdesc} depends only on $A=g^{-1}\cdot 0$; equation 
\eqref{eq:otherat} gives a direct description in terms of $A$.  (Recall that $\partial_A=\Ad_{g^{-1}}\circ \ \partial\circ \Ad_g$.) The expression for the Lie bracket follows from a similar formula for the bracket 
on sections of $T(\P G)$.  
\end{proof}

\subsection{Reduction by the $L_0G$-action}\label{redbaslop}
The lift of the $LG$-action on $\A_{S^1}$ to the standard Courant algebroid $\T \A_{S^1}$ has isotropic generators $\varrho\colon L\g\to \Gamma(\T \A_{S^1})$ given by the same formulas 
as for $\A_\II$:
\begin{equation}\label{eq:varrhoxi}
 \varrho(\xi)=\xi_{\A_{S^1}}+\l\d A,\,\xi\r,\ \ \xi\in L\g.
 \end{equation}
By \eqref{eq:cova},  the fiber of $\E=\varrho(\A_{S^1}\times L\g)$ at $A\in \A_{S^1}$ may be regarded as the graph of the skew-adjoint operator $\partial_A\colon \Omega^0(S^1,\g)\to \Omega^1(S^1,\g)$. In particular, $\E$ is a Lagrangian subbundle, and since it is involutive it is a Dirac structure 
$\E\subset \T \A_{S^1}$. Indeed, $\E$ is a weak Poisson structure, which we will again refer to as a Lie-Poisson structure on $\A_{S^1}$. 

To describe its reduction with respect to the based loop group $L_0G$, we extend \eqref{eq:varrhoxi} to sections  of the Lie algebroid $R$: 
\begin{equation}\label{eq:varrhoxi1}
 \varrho(\xi)=\xi_{\A_{S^1}}+\l\d A,\,\xi|_\II \r,\ \ \xi\in \Gamma(R); 
 \end{equation}
here we denote by $\xi|_\II\in \Omega^0_{\HH_{-r}}(S^1,\g)$ the restriction 
to $\II\subset \R$, regarded as a piecewise continuous function on $S^1$ (with a jump singularity at $0$) and by  $\l\d A,\,\xi|_\II\r$ the corresponding element of $T_A^\ast \A_{S^1}$. 
\begin{lemma}
For $\xi_1,\xi_2\in \Gamma(R)$, the pairing of  the corresponding sections is given by 
\[ \l \varrho(\xi_1),\varrho(\xi_2)\r=\xi_1(1)\cdot \xi_2(1)-\xi_1(0)\cdot \xi_2(0),\]
while the Courant bracket is
\[ \Cour{\varrho(\xi_1),\varrho(\xi_2)}=\varrho([\xi_1,\xi_2]_R)+
\xi_2(1)\cdot \d \xi_1(1)-\xi_2(0)\cdot \d \xi_1(0)
.\]
\end{lemma}
\begin{proof}
We will write $\xi^\sharp=\xi_{\A}$ for the vector field defined by $\xi\in \Gamma(R)$. 
The formula for the pairing follows from 
\[ \iota_{\xi_1^\sharp}\ \l\d A,\,\xi_2\r
+\iota_{\xi_2^\sharp}\ \l\d A,\,\xi_1\r
=\l \partial_A\xi_1,\xi_2\r+\l \xi_1,\partial_A \xi_2\r
=\int_\II \partial(\xi_1\cdot\xi_2),
\]
The vector field component of the Courant bracket $\Cour{\varrho(\xi_1),\varrho(\xi_2)}$ 
is $[\xi_1^\sharp,\xi_2^\sharp]=[\xi_1,\xi_2]_R^\sharp$. 
For the 1-form component, we have to calculate 
\[ \begin{split}
\L_{\xi_1^\sharp}(\l\d A,\,\xi_2\r)-\iota_{\xi_2^\sharp}\d (\l\d A,\,\xi_1\r)
&=\l\d A,\,(\L_{\xi_1^\sharp}\xi_2-\L_{\xi_2^\sharp}\xi_1)\r
+\l \d \L_{\xi_1^\sharp} A,\, \xi_2\r+\l\iota_{\xi_2^\sharp}\d A,\, \d\xi_1\r\\
&=\l\d A,\, (\L_{\xi_1^\sharp}\xi_2-\L_{\xi_2^\sharp}\xi_1)\r
+\l \d \partial_A \xi_1,\, \xi_2\r+\l\partial_A\xi_2,\, \d\xi_1\r\\
\end{split}\]
But $\l\d \partial_A \xi_1,\, \xi_2\r=\l \partial_A \d \xi_1,\, \xi_2\r+\l[\d A,\xi_1],\,\xi_2\r$. 
The first term combines with $\l\partial_A\xi_2,\,\d\xi_1\r$ to give 
$\int_\II \partial(\xi_2\cdot \d \xi_1)$, 
while the second term combines with $\l\d A,\, (\L_{\xi_1^\sharp}\xi_2-\L_{\xi_2^\sharp}\xi_1)\r$ to 
$\l\d A,\, [\xi_1,\xi_2]_R\r$.  
\end{proof}
We are now in position to compute the reduction of the Lie-Poisson structure $\E\subset \T \A_{S^1}$ by the action of $L_0G$. By definition, the reduced Courant algebroid is $(C/C^\perp)/L_0G$, where 
$C$ is the coisotropic subbundle with fibers
$C_A=(\varrho(L_0\g)_A)^\perp$. 

\begin{theorem}[Reduction of the weak Poisson structure on $\A_{S^1}$] 
The reduction of the Dirac structure $(\T \A_{S^1},\E)$ under the action of  the based loop group 
$L_0G$ is canonically isomorphic to the Cartan-Dirac structure $(\AA,E)$. In more detail, 
$C$ is spanned by sections $\varrho(\xi)$ with $\xi\in \Gamma(R)$, and the map 
\[ C\to G\times (\ol{\g}\oplus \g),\ \ \varrho(\xi)\mapsto (\Hol(A), \xi(0),\xi(1))\]
descends to an isomorphism  of Courant algebroids $(\T \A_{S^1})_\red\to \AA$.  
\end{theorem}
\begin{proof}
An  element $a+\l\d A,u\r$ with $a\in T_A\A_{S^1}=\Omega^1_{\HH_r}(S^1,\g)$ and $u\in 
\Omega^0_{\HH_{-r}}(S^1,\g)$, lies in $C_A=\varrho(L_0\g)_A^\perp$ if and only if for all 
$\tau\in L_0\g$, 
\[ 0=\big\l a+\l\d A,u\r,\ \partial_A\tau+\l\d A,\tau\r
\big\r=\big\l a-\partial_A u,\ \tau
\big\r
\]
Equivalently, $a-\partial_A u$ is a multiple of the $\delta$-distribution supported at $0$. In particular, $u$ is given by a continuous function on $\II$ (regarded as a  piecewise continuous function 
on $S^1$ with a jump discontinuity at $0$). Given $a\in T_A\A_{S^1}=\Omega^1_{\HH_r}(S^1,\g)$, 
we can determine the corresponding $u$ by integration. Furthermore, 
by lifting the differential equation to $\R$, we see that $u$ is the restriction to $\II$ 
of a function $\xi\in \Omega^0_{\HH_{r+1}}(\R,\g)$ satisfying $\partial_A\xi=a$ (where $A,a$
are regarded as periodic forms on  $\R$). In particular, $\partial_A\xi$ is periodic, that is, $\xi\in R_A$. This gives the desired identification of $R_A\to C_A,\ \xi\mapsto \varrho(\xi)_A$. 

Since the kernel of the map $R_A\to \ol{\g}\oplus \g,\ \ \xi\mapsto (\Hol(A),\xi(0),\xi(1))$ is exactly $L_0\g$, it follows that $(\T\A_{S^1})_{\red}=G\times (\ol{\g}\oplus\g)$ as a vector bundle. 
The metric and Courant bracket on $(\T\A_{S^1})_{\red}$ are induced from the metric and Courant bracket on $L_0G$-invariant sections of $C$; using the Lemma we obtain the metric and Courant bracket of the Cartan-Courant algebroid. Finally, since the $L_0G$-invariant sections $\varrho(\xi)$ 
of $\E\subset R$ are those with $\xi(0)=\xi(1)$, we see that $\E_\red$ is the Cartan-Dirac structure. 
\end{proof}

Similar to $\A_\II$, the fibration $\A_{S^1}\to G$ has a standard connection, defined by any choice of a function $\chi\in C^\infty(\II)$  such that $\chi$ extends to a smooth function 
on $\R$, equal to $0$ for $t\le 0$ and equal to $1$ for $t\ge 1$. The connection is best described in terms of the caloron correspondence, Appendix \ref{app:cal}. Arguing as in the case of $\A_\II$, we obtain:   
\begin{theorem}
The reduction of the Dirac structure $(\T \A_{S^1},\E)$ with respect to the based loop group $L_0G$ is $G=LG/L_0G$-equivariantly isomorphic to the Cartan-Dirac structure $(\AA,E)$ over $G$. Furthermore, the reduction of the $L_0G$-basic splitting of $\T \A_{S^1}$, defined by the standard connection $\theta$ on the holonomy fibration,  is the usual splitting of  the Cartan-Courant algebroid, identifying $\AA\cong \T G_\eta$. The reduction procedure gives a one-to-one correspondence between $LG$-equivariant (exact) Hamiltonian spaces for $(\T \A_{S^1},\E)$ and $G$-equivariant (exact) Hamiltonian spaces for $(\AA,E)$. 
\end{theorem}

\begin{appendix}
\section{Reduction in infinite dimensions}\label{sec:appinfty}
  Let $V$ be a Banach space. The closure of a subspace $F\subseteq V$
  will be denoted $\on{cl}(F)$, and the annihilator
  $\on{ann}(F)\subseteq V^\star $, where $V^\star $ is the topological
  dual space of $V$. For Banach spaces $V,\,V'$, denote by
  $\mathbb{B}(V,V')$ the Banach space of continuous linear maps $V\to
  V'$. More generally, given Banach spaces $V_1,\ldots,V_l$ there is a
  Banach space $\mathbb{B}(V_1,\ldots,V_l;V')$ of continuous
  multilinear maps $V_1\times\cdots\times V_l\to V'$.

  Suppose $V$ is a Hilbert space with a pseudo-Riemannian metric
  $B$. Let $B^\flat \colon V\to V^\star $ be the associated map. For
  any subspace $F\subseteq V$, we have $B^\flat(F^\perp)=\on{ann}(F)$,
  and $(F^\perp)^\perp=\on{cl}(F)$.

  For the following Proposition, we observe that if $F_1,F_2$ are
  closed subspace of a real Hilbert space $V$, then $F_1+F_2$ is
  closed in $V$ if and only if $\on{ann}(F_1)+\on{ann}(F_2)$ is closed
  in $V^\star $.  (Proof: let $F_1',F_2'$ be closed complements to
  $F_1\cap F_2$ in $F_1,F_2$ respectively. If $F_1+F_2$ is closed, let
  $N$ be a closed complement to $F_1+F_2$ in $V$. Then $V=F_1\cap
  F_2\oplus F_1'\oplus F_2'\oplus N$ is a direct sum decomposition of
  $V$ into closed subspaces. By considering the dual decomposition of
  $V$, it follows that the inclusion $\on{ann}(F_1)+\on{ann}(F_2)\to
  \on{ann}(F_1\cap F_2)$ is an equality.)

  Thus, if $V$ carries a metric $B$, then $F_1+F_2$ is closed if and
  only if $F_1^\perp+F_2^\perp$ is closed. Criteria for $F_1+F_2$ to
  be closed may be found in \cite{schoch:clo}; in particular, it is
  known that the sum of disjoint closed subspaces is closed if and
  only if a suitably defined `angle' between these subspaces is
  non-zero.

\begin{proposition}\label{prop:linred}
  Let $V$ be a real Hilbert space with a metric $B$, and $C$ a closed
  co-isotropic subspace of $V$. Then
  \begin{enumerate}
  \item $C$ admits a closed isotropic complement. (In particular,
    every Lagrangian subspace admits a Lagrangian complement.)
  \item The quotient $V_C=C/C^\perp$ inherits a metric $B_C$,
  \item Suppose $L\subseteq V$ is Lagrangian. Then $L+C$ is closed if
    and only if $L+C^\perp$ is closed, and in this case $L_C=(L\cap
    C)/(L\cap C^\perp)$ is Lagrangian in $V_C$.
  \end{enumerate}
\end{proposition}
\begin{proof}
  \begin{enumerate}
  \item Choose a closed complement $F$ to $C$. Then $F^\perp$ is a
    closed complement to $C^\perp$. The projection to $C^\perp$ along
    $F^\perp$ restricts to a continuous linear map $A\colon F\to
    C^\perp$, and
    \[ F'=\{v-\hh A(v)|\ v\in F\} \] is the desired isotropic
    complement to $C$. ($F'$ is closed since it is the graph of a
    continuous linear map $-\hh A\colon F\to C^\perp\subseteq C$.)
  \item The bilinear form $B$ descends to a continuous symmetric
    bilinear form $B_C\colon V_C\times V_C\to R$. We have to verify
    that $B_C$ is non-degenerate. Let $F$ be a closed isotropic
    subspace with $V=C\oplus F$, hence $V=C^\perp\oplus
    F^\perp$. Intersecting with $C$, it follows that $C=C^\perp\oplus
    (C\cap F^\perp)$, thus $V=C^\perp \oplus F\oplus (C\cap
    F^\perp)$. The quotient map $C\to V_C$ induces a topological
    isomorphism $C\cap F^\perp\to V_C$, identifying $B_C$ with the
    restriction of $B$ to $C\cap F^\perp=(C^\perp \oplus F)^\perp$.
    The latter is non-degenerate, hence so is $B_C$.
  \item The inverse image of $L_C^\perp$ in $C$ is
    \[ (L\cap C)^\perp\cap C =\on{cl}(L+C^\perp)\cap C\supseteq
    (L+C^\perp)\cap C=(L\cap C)+C^\perp.\]
    Applying the projection $C\to V_C$, it follows that
    $L_C^\perp\supseteq L_C$. If $L+C^\perp$ is closed, the inclusion
    becomes an equality, and we obtain $L_C^\perp=L_C$.
  \end{enumerate}
\end{proof}
%
%

\section{Lifting problems}\label{sec:brylinski}
  Let $Q\to B$ be a principal $G$-bundle, and
  \[ 
  1\to \U(1)\to \wh{G}\to G\to 1
  \]
  a central extension. Consider the exact sequence of vector bundles
  over $B$,
  \begin{equation}\label{eq:exseq} 
    0\to B\times\R\to Q \times_G{\wh{\g}}\to Q\times_G\g\to 0.
  \end{equation}
  A splitting of this sequence may be regarded as a $G$-equivariant
  map $\nu\colon \g\to \Omega^0(Q,\wh{\g})$ whose composition with the
  projection $\wh{\g}\to \g$ is the identity. The differential of this map is scalar-valued, defining 
  a linear map 
  \[\alpha\colon \g\to \Omega^1_{\on{cl}}(Q),\ \xi\mapsto \d \nu(\xi)\]
  with values in closed 1-forms. The map 
  \[ \varrho\colon \g\to \Gamma(\T Q),\ \xi\mapsto \xi_Q+\alpha(\xi)\] gives isotropic generators for the natural $G$-action on
  $\T Q$. The standard splitting of $\T Q$ is not basic for this $G$-action. 
  However, by Proposition \ref{prop:basplit} any principal connection $\theta$ on $Q$ defines a new
  $G$-basic splitting of $\T Q$, giving an identification $(\T
  Q)_\red=\T B_\eta$ for a closed 3-form $\eta\in\Omega^3(B)$. The
  construction also gives a 2-form $\varpi$ on $Q$ with
  $\d\varpi=-\pi^*\eta$. These are exactly 
   the 2-form and 3-form
  appearing in Brylinski's discussion of the problem of lifting the
  structure group to $\wh{G}$ \cite{br:lo}. In particular, the
  cohomology class of $\eta$ is the image in de Rham cohomology of the
  obstruction class in $H^3(B,\Z)$ for the existence of a lift.
\section{Caloron correspondence}\label{app:cal}
The \emph{caloron correspondence}, due to  Garland-Murray \cite{gar:kac}, Murray-Stevenson \cite{mur:hig}, and Murray-Vozzo \cite{mur:cal}, relates principal bundles over a base $B$, with structure group the (based) loop group, with (framed) principal bundles over a base $B\times S^1$, with structure group $G$. Among other things, this correspondence leads to a simple construction of principal connections on the loop group bundle. 
\subsection{Caloron correspondence for $\A_\II$}
In this section we will use  a version of the caloron correspondence where we work with path spaces rather than loop spaces. 
A \emph{framing} of a principal $G$-bundle $Q\to B$ along a submanifold $Z\subseteq B$ is a trivialization along $Z$, i.e., a section $\sigma\colon Z\to Q|_{Z}$. A principal connection $\nu\in \Omega^1(Q,\g)$ is a \emph{framed connection} if $\sigma^*\nu=0$. 
Given a manifold $M$ with two submanifolds $M_0,M_1$, we say that $\gamma\colon \II\to M$ is a \emph{based path} if $\gamma(0)\in M_0$ and $\gamma(1)\in M_1$. Let $M_\II$ be the space of 
paths $\II\to M$ of Sobolev class $r+1$, and $M_{\II,\partial \II}\subseteq M_\II$ the based paths. Given a principal bundle as above, with framings $\sigma_i\colon B_i\to Q$ along $B_i\subseteq B$, and taking $Q_i=\sig_i(B_i)$, 
we obtain a diagram of principal bundles 
\[ \xymatrix
{Q_{\II,\partial \II} \ar[r]\ar[d]^{/ G_{\II,\partial \II}} & Q_\II\ar[d]^{/G_\II}\\  B_{\II,\partial\II} \ar[r] & B_\II
}\]
Any principal connection $\nu\in \Omega^1(Q,\g)$ determines a principal connection $\nu_\II$ on the 
bundle $Q_\II$. If $\nu$ is a framed connection, then $\nu_\II$ restricts to a principal connection 
on $Q_{\II,\partial\II}$. 

As a special case, take $Q$ to be the \emph{trivial} principal $G$-bundle $Q=B\times G$ over $B=G\times\II$, with the framings along $B_0=G\times \{0\},\ B_1=G\times \{1\}$ given by
\[ \sig_0(a,0)=(a,0,e),\ \ \ \sig_1(a,1)=(a,1,a),\]
and with the principal $G$-action $k.(a,s,g)=(a,s,gk^{-1})$. Consider the inclusion 
$G\to B_{\II,\partial\II}$, taking $a\in G$ to the path $\gamma(t)=(a,t)$. 
The restriction of 
$Q_{\II,\partial \II} $ to this submanifold $G\subseteq B_{\II,\partial\II}$
is identified with 
$G_{\II,0}=\{g\in G_\II | g(0)=e\}$, by the map 
\[ G_{\II,0}\to Q_{\II,\partial \II},\ g\mapsto \big(t\mapsto (g(1),t,g(t))\big).\]
On the other hand, the map $G_\II\to \A_\II,\ g\mapsto g^{-1}\cdot 0$ restricts to a diffeomorphism
 $G_{\II,0}\cong \A_\II$.
In summary, we have a commutative diagram, 
\[ \xymatrix
{
\A_\II \ar[r]\ar[d]^{/G_{\II,\partial\II}} & (G\times\II\times G)_{\II,\partial \II}\ar[d]^{/G_{\II,\partial\II}}
\ar[r] & (G\times \II\times G)_\II \ar[d]^{/G_\II}
\\
G \ar[r] & (G\times \II)_{\II,\partial \II}\ar[r] & (G\times \II)_\II
}
\]
To incorporate the $G_\II$-action on $\A_\II$ in this picture, note that the principal action of $G$ on 
$Q$ extends to an action of $G\times G\times G$:
\[ (u,v,k).(a,s,g)=(u a v^{-1},s,ugk^{-1}).\]
It defines a $G_\II \times G_\II\times G_\II$-action on $Q_\II$, given by the same formula 
(but with $u$, $k$, etc.~ as paths). The subbundle $Q_{\II,\partial \II}$ is preserved by the subgroup of paths $(u,v,k)$ such that $u(0)=k(0)$ and $v(1)=k(1)$, and the subbundle $\A_\II$ by the subgroup 
$G_\II\subseteq G_\II\times G_\II\times G_\II$ of paths of the form $(u,v,k)(t)=(k(0),k(1),k(t))$. 

As explained above, a framed principal connection $\nu$ on $Q$ defines a principal connection $\nu_\II$ on $Q_\II$, which then pulls back to a connection on $Q_{\II,\partial\II}$. Let $\theta$ denote its 
restriction to $\A_\II$.  If $\nu$ is furthermore invariant under the action of $(u,v)\in G\times G$ by automorphisms, then $\nu_\II$ will be invariant under the $G_\II\times G_\II$-action. That is, the horizontal subbundle $\ker(\nu_\II)\subseteq TQ_\II$ is invariant not just under the gauge action, but 
under the full $G_\II\times G_\II\times G_\II$-action. It then follows that the connection $\theta$ is $G_\II$-equivariant, in the sense that the horizontal distribution $\ker(\theta)$ is $G_\II$-invariant. 

To get concrete formulas, we express the principal connection $\nu$ on $Q=B\times G$ in terms of its connection 1-forms $\kappa\in \Omega^1(B,\g)$:
\[ \nu=\Ad_{g^{-1}}\kappa+g^*\theta^L.\]
Here the variable $g$ is regarded as the projection $g\colon B\times G\to G$). The connection $\nu$ is a framed connection if and only if 
\begin{equation}\label{eq:framed}
 i_0^*\kappa=0,\ \ \ i_1^*\kappa=-a^*\theta^R,
\end{equation}
where $i_s\colon G\to B,\ a\mapsto (a,s)$. 
It is furthermore invariant under the $G\times G$-action by automorphisms if and only if 
\begin{equation}\label{eq:kinvar}
(u,v)^*\kappa=\Ad_u \kappa.\end{equation}
\begin{proposition} 
Let $\nu$ be a framed connection on $Q=B\times G$, defined by a connection 1-form  
$\kappa\in \Omega^1(G\times \II,\g)$. For $t\in \II$, let $\kappa_t=i_t^*\kappa$.  
Let $A\in \A_\II$, defining  a parallel transport $g\in G_{\II,0}$. 
Then the horizontal lift for the resulting connection 1-form $\theta$ is given at $A\in\A_\II$ by 
\[ T_{\Hol(A)}G\to T_A\A_\II,\ \ X\mapsto \partial_A \xi\]
where $\xi\in\g_\II$ is the path $\xi(t)=-\Ad_{g(t)^{-1}}\kappa_t(X)$. 
\end{proposition}
\begin{proof}
Note that $\xi(0)=0$, while 
\[ \xi(1)=-\Ad_{\Hol(A)^{-1}}\kappa_1(X)=\iota(X)\theta^L|_{\Hol(A)}.\]  
The proposition
asserts that the horizontal lift of $X$ is given by $\partial_A\xi\in T_A\A_\II$,  the infinitesimal action of $\xi\in\g_\II$ on $\A_\II$. The image of $\partial_A\xi$ under the differential of the map 
$\A_\II\to (G\times \II\times G)_\II$ is the infinitesimal action of $(\xi(0),\xi(1),\xi)\in \g_\II\times \g_\II\times \g_\II$ at $(\Hol(A),0,g)$, that is, 
\begin{equation}\label{eq:so}
 \big(\xi(1)^L|_{\Hol(A)},\ 0,\ \xi_{G_\II}|_g\big)=(X,0,\xi_{G_\II}|_g\big).
 \end{equation}
On the other hand,  the image of $X\in T_{\Hol(A)}G$ under the differential of $G\to (G\times \II)_\II$  is the constant vector field $(X,0)\in T B_\II$, and by the formula for $\nu_\II$ in terms of the connection 1-form, \eqref{eq:so} is precisely the horizontal lift of $(X,0)$. 
\end{proof}


A convenient choice for $\kappa$ satisfying \eqref{eq:framed} as well as the invariance 
\eqref{eq:kinvar} is given by 
\begin{equation}\label{eq:nicechoice}
\kappa=-\chi(s)\ a^*\theta^R
\end{equation}
for any function $\chi\in C^\infty(\II)$ such that $\chi(0)=0,\ \chi(1)=1$.

\subsection{Caloron correspondence for $\A_{S^1}$}
The caloron correspondence for $\A_{S^1}$ runs as follows (see \cite[Example 3.4]{mur:cal}). 
Consider the trivial principal $G$-bundle $\wt{Q}=G\times \R\times G$, with the principal action of $x\in G$ 
given as
\[ x\cdot (a,s,y)=(a,s,y x^{-1}),\] 
for $a,y\in G$ and $s\in\R$. The group of integers $\Z$ acts by principal bundle automorphisms, 
$n\cdot (a,s,y)=(a,s+n,a^n y)$; the quotient is a principal bundle 
\[  Q=(G\times \R\times G)/\Z\to G\times S^1,\ \ [(a,s,y)]\mapsto (a,[s]),\]
with a canonical framing along $G\times \{[0]\}$, given by 
$(a,[0])\mapsto [(a,0,e)]$, and with a $G\times \R$-action by bundle automorphisms
\[ (a',s').[(a,s,y)]=[(\Ad_{a'}a,\ s+s',\ a'a)].\]
 Taking loops of Sobolev class $r+1$, we obtain a $G$-equivariant principal $LG$-bundle $ LQ\to L(G\times S^1)$, containing the  bundle of quasi-periodic paths $\ca{P}G$ as a $G$-equivariant subbundle:
\[ \xymatrix
{
\ca{P}G \ar[r]\ar[d]_\pi & LQ\ar[d] \\
G \ar[r] & L(G\times S^1)
}
\]
Here the lower horizontal map takes $a\in G$ to the loop 
$s\mapsto (a,s)$, while the upper horizontal map takes $g\in \ca{P}G$ to the loop, 
$s\mapsto [(\pi(g),s,g(s))]$. Similarly, working with framed loops we obtain a diagram
\[ \xymatrix
{
\A_{S^1} \ar[r]\ar[d]_{\Hol}& L_0Q\ar[d] \\
G \ar[r] & L_0(G\times S^1)
}
\]

Given a principal connection $\nu\in \Omega^1(Q,\g)$ on the bundle $ Q\to G\times S^1$, 
the loop functor determines a connection on $ LQ\to L(G\times S^1)$, which then pulls back to a connection $\theta$ on the principal $LG$-bundle $\ca{P}G\to G$. Furthermore, if $\nu$ is a \emph{framed connection}, then the resulting connection on $LQ\to LG$ restricts to a connection 
on $L_0Q$, and hence $\theta$ reduces to a connection on $\A_{S^1}\cong \ca{P}_0G\to G$.

To describe framed connections on $Q$, we use  the canonical trivialization 
of its pullback under the map $G\times \II\to G\times S^1,\ (a,s)\mapsto (a,[s])$. 
A sufficient condition for  $\chi\in C^\infty(\II)$ with $\chi(0)=0,\ \chi(1)=1$ to 
define a connection on $Q$, is that $\chi$ extends to a smooth function on $\R$, equal to $0$ for 
for $s\le 0$ and equal to $1$ for $s\ge 1$. The resulting connection $\theta$ on the loop group bundle $\ca{P}G\to G$ is again referred to as a standard connection. Connections of this type were used by Carey-Mickelsson 
\cite{car:uni}. While $\theta$ depends on the choice of $\chi$, the resulting 2-form $\varpi\in \Omega^2(\A_{S^1})$ is independent of that choice \cite{al:ati}; it is the pullback of the 
corresponding 2-form on $\A=\A_\II$. 
\end{appendix}

\def\cprime{$'$} \def\polhk#1{\setbox0=\hbox{#1}{\ooalign{\hidewidth
  \lower1.5ex\hbox{`}\hidewidth\crcr\unhbox0}}} \def\cprime{$'$}
  \def\cprime{$'$} \def\cprime{$'$} \def\cprime{$'$}
  \def\polhk#1{\setbox0=\hbox{#1}{\ooalign{\hidewidth
  \lower1.5ex\hbox{`}\hidewidth\crcr\unhbox0}}} \def\cprime{$'$}
  \def\cprime{$'$} \def\cprime{$'$} \def\cprime{$'$} \def\cprime{$'$}
\providecommand{\bysame}{\leavevmode\hbox to3em{\hrulefill}\thinspace}
\providecommand{\MR}{\relax\ifhmode\unskip\space\fi MR }
\providecommand{\MRhref}[2]{%
  \href{http://www.ams.org/mathscinet-getitem?mr=#1}{#2}
}
\providecommand{\href}[2]{#2}


\begin{thebibliography}{10}

\bibitem{ab:ma}
R.~Abraham, J.~Marsden, and T.~Ratiu, \emph{Manifolds, tensor analysis and
  applications}, Addison-Wesley, Reading, 1983.

\bibitem{al:pur}
A.~Alekseev, H.~Bursztyn, and E.~Meinrenken, \emph{Pure spinors on {L}ie
  groups}, Ast\'{e}risque \textbf{327} (2009), 131--199.

\bibitem{al:mom}
A.~Alekseev, A.~Malkin, and E.~Meinrenken, \emph{{L}ie group valued moment
  maps}, J.~Differential Geom. \textbf{48} (1998), no.~3, 445--495.

\bibitem{al:ati}
E.~Alekseev, A.~and~Meinrenken, \emph{The {A}tiyah algebroid of the path
  fibration over a {L}ie group}, Lett.~ Math.~ Phys.~ \textbf{90} (2009),
  23--58.

\bibitem{at:mo}
M.~F. Atiyah and R.~Bott, \emph{The {Y}ang-{M}ills equations over {R}iemann
  surfaces}, Phil.~Trans.~Roy.~Soc.~London Ser.~A \textbf{308} (1982),
  523--615.

\bibitem{bal:not}
R.~Balan, \emph{A note about integrability of distributions with
  singularities}, Boll. Un. Mat. Ital. A (7) \textbf{8} (1994), no.~3,
  335--344. 

\bibitem{bo:el}
B.~Boo{\ss}-Bavnbek and K.~P. Wojciechowski, \emph{Elliptic boundary value
  problems for {D}irac operators}, Birkh\"auser, Boston, 1993.

\bibitem{br:lo}
J.-L. Brylinski, \emph{Loop spaces, characteristic classes and geometric
  quantization}, Birkh\"auser Boston Inc., Boston, MA, 1993.

\bibitem{bur:red}
H.~Bursztyn, G.~Cavalcanti, and M.~Gualtieri, \emph{Reduction of {C}ourant
  algebroids and generalized complex structures}, Adv.~Math. \textbf{211}
  (2007), no.~2, 726--765.

\bibitem{bur:di}
H.~Bursztyn and M.~Crainic, \emph{Dirac structures, momentum maps, and
  quasi-{P}oisson manifolds}, The breadth of symplectic and Poisson geometry,
  Progr.~Math., vol. 232, Birkh\"auser Boston, Boston, MA, 2005, pp.~1--40.

\bibitem{bur:cou}
H.~Bursztyn, D.~Iglesias Ponte, and P.~Severa, \emph{Courant morphisms and
  moment maps}, Math. Res. Lett. \textbf{16} (2009), no.~2, 215--232.

\bibitem{car:uni}
A.L. Carey and J.~Mickelsson, \emph{{The universal gerbe, Dixmier-Douady class,
  and gauge theory}}, Lett.Math.Phys. \textbf{59} (2002), 47--60.

\bibitem{car:fu}
A.L. Carey and B.-L. Wang, \emph{Fusion of symmetric $d$-branes and {V}erlinde
  rings}, Comm.~Math.~Phys. \textbf{277} (2008), 577--625.

\bibitem{chi:int}
D.~Chillingworth and P.~Stefan, \emph{Integrability of singular distributons on
  {B}anach manifolds}, Math. Proc. Camb. Phil. Soc. \textbf{79} (1976),
  117--128.

\bibitem{cou:di}
T.~Courant, \emph{Dirac manifolds}, Trans.~Amer.~Math.~Soc. \textbf{319}
  (1990), no.~2, 631--661.

\bibitem{st:db}
J.~M. Figueroa-O'Farrill and S.~Stanciu, \emph{D-brane charge, flux
  quantisation and relative (co)homology}, J.~High Energy Phys. (2001), no.~1,
  Paper 6, 16.

\bibitem{fre:ins}
D.~Freed and K.~Uhlenbeck, \emph{Instantons and four-manifolds}, Mathematical
  Sciences Research Institute Publications, vol.~1, Springer-Verlag, New York,
  1984.

\bibitem{gar:kac}
H.~Garland and M.~Murray, \emph{Kac-{M}oody monopoles and periodic instantons},
  Comm. Math. Phys. \textbf{120} (1988), no.~2, 335--351.

\bibitem{gua:ge1}
M.~Gualtieri, \emph{Generalized complex geometry}, Ann. of Math. (2)
  \textbf{174} (2011), no.~1, 75--123.

\bibitem{kli:wzw}
T.~Klim{\v{c}}{\'{\i}}k, C.~and~Strobl, \emph{W{ZW}-{P}oisson manifolds},
  J.~Geom.~Phys. \textbf{43} (2002), no.~4, 341--344.

\bibitem{lib:cou}
D.~Li-Bland and E.~Meinrenken, \emph{{C}ourant algebroids and {P}oisson
  geometry}, International Mathematics Research Notices \textbf{11} (2009),
  2106--2145.

\bibitem{lib:dir}
\bysame, \emph{{Dirac Lie groups}}, Asian Journal of Mathematics \textbf{18}
  (2014), no.~5, 779--816.

\bibitem{lib:qua}
D.~Li-Bland and P.~\v{S}evera, \emph{Quasi-{H}amiltonian groupoids and
  multiplicative {M}anin pairs}, International Mathematics Research Notices
  \textbf{2011} (2011), 2295--2350.

\bibitem{lin:com}
J.~Lindenstrauss and L.~Tzafriri, \emph{On the complemented subspaces problem},
  Israel J. Math. \textbf{9} (1971), 263--269. 

\bibitem{liu:ma}
Z.-J. Liu, A.~Weinstein, and P.~Xu, \emph{Manin triples for {L}ie
  bialgebroids}, J.~Differential Geom. \textbf{45} (1997), no.~3, 547--574.

\bibitem{me:lo}
E.~Meinrenken and C.~Woodward, \emph{{H}amiltonian loop group actions and
  {V}erlinde factorization}, J.~Differential Geom. \textbf{50} (1999),
  417--470.

\bibitem{mur:hig}
M.~Murray and D.~Stevenson, \emph{Higgs fields, bundle gerbes and string
  structures}, Comm. Math. Phys. \textbf{243} (2003), no.~3, 541--555.

\bibitem{mur:cal}
M.~Murray and R.~Vozzo, \emph{The caloron correspondence and higher string
  classes for loop groups}, J. Geom. Phys. \textbf{60} (2010), no.~9,
  1235--1250.

\bibitem{odz:ban}
A.~Odzijewicz and T.~Ratiu, \emph{Banach {L}ie-{P}oisson spaces and reduction},
  Comm. Math. Phys. \textbf{243} (2003), no.~1, 1--54.

\bibitem{pel:int}
F.~Pelletier, \emph{Integrability of weak distributions on {B}anach manifolds},
  Indag. Math. \textbf{23} (2012), no.~3, 214--242.

\bibitem{po:qu}
V.~L. Popov, \emph{Quasihomogeneous affine algebraic varieties of the group
  {$\operatorname{SL}(2)$}}, Math.~USSR-Izv. \textbf{7} (1973), no.~4,
  793--831.

\bibitem{roy:co}
D.~Roytenberg, \emph{{Courant algebroids, derived brackets and even symplectic
  supermanifolds}}, Thesis, Berkeley 1999. arXiv:math.DG/9910078.

\bibitem{schoch:clo}
I.~Schochetman, R.~Smith, and S.-K. Tsui, \emph{On the closure of the sum of
  closed subspaces}, Int. J. Math. Math. Sci. \textbf{26} (2001), no.~5,
  257--267.

\bibitem{sev:let}
P.~{\v{S}}evera, \emph{Letters to {A}lan {W}einstein},
  http://sophia.dtp.fmph.uniba.sk/~severa/letters/, 1998-2000.

\bibitem{ste:int}
P.~Stefan, \emph{Integrability of systems of vector fields}, J. London Math.
  Soc. (2) \textbf{21} (1980), no.~3, 544--556. 

\bibitem{sus:orb}
H.~Sussmann, \emph{Orbits of families of vector fields and
  integrability of distributions}, Trans. Amer. Math. Soc. \textbf{180} (1973),
  171--188.  

\bibitem{uch:rem}
K.~Uchino, \emph{Remarks on the definition of a {C}ourant algebroid}, Lett.
  Math. Phys. \textbf{60} (2002), no.~2, 171--175.

\end{thebibliography}
\end{document}